\documentclass[11pt]{amsart}
\usepackage{amsmath}
\usepackage[dvips]{epsfig}
\usepackage{slashed}
\usepackage{accents}
\usepackage{color}
\usepackage{leftidx}

\usepackage{graphicx}
\usepackage{hyperref}

\newtheorem{theorem}{Theorem}[section]
\newtheorem{lemma}[theorem]{Lemma}
\newtheorem{coro}[theorem]{Corollary}
\newtheorem{pro}[theorem]{Proposition}

\theoremstyle{definition}

\theoremstyle{remark}

\numberwithin{equation}{section}

\newcommand{\theoremref}[1]{\hyperref[#1]{Theorem~\ref*{#1}}}
\newcommand{\lemmaref}[1]{\hyperref[#1]{Lemma~\ref*{#1}}}
\newcommand{\proref}[1]{\hyperref[#1]{Proposition~\ref*{#1}}}
\newcommand{\definitionref}[1]{\hyperref[#1]{Definition~\ref*{#1}}}
\newcommand{\cororef}[1]{\hyperref[#1]{Corollary~\ref*{#1}}}

\begin{document}

\title[Local isometric embedding in $\mathbb{R}^3$]{On the local isometric embedding in $\mathbb{R}^3$ of surfaces with zero sets of Gaussian curvature forming cusp domains}

\author{Tsung-Yin Lin}




\begin{abstract}
We study the problem of isometrically embedding a two-dimensional Riemannian manifold into Euclidean three-space. It is shown that if the Gaussian curvature vanishes to finite order and its zero set consists of two smooth curves tangent at a point, then local sufficiently  smooth isometric embeddings exists.
\end{abstract}

\maketitle



\section{Introduction}\label{intro}
It is a natural question to ask whether all two-dimensional Riemannian manifolds admit local isometric embeddedings into $\mathbb{R}^3$, or stated in another way, whether every abstractly defined surface actually arises (at least locally) as a concrete surface that we may visualize in 3-space. The purpose of this paper is to provide a sufficient condition for the existence of such embeddings when the zero set of the Gaussian curvature possesses cusp intersections. 
This will be accomplished by solving the traditional Monge-Amp\'ere equation, some times referred to as the Darboux equation, that is associated with this problem.

We point out first that counterexamples to the existence of local isometric embeddings were constructed for metrics of low regularity, by Pogorelov \cite{MR0286034} and by Nadirashvili and Yuan \cite{MR2393070}, both for metrics $g\in C^{2,1}$. Moreover, Khuri found counterexamples to the local solvability of smooth Monge-Amp\'ere equations in \cite{MR2334827}. Yet the question of whether or not there are any smooth counterexamples to the isometric embedding problem is still open.

On the other hand, affirmative answers have been given for the cases when the Gaussian curvature $K$ does not vanish or when the metric is analytic. These classical results may be proven by standard implicit function theorem arguments when $K\neq 0$, and by the Cauchy-Kowalevski theorem in the analytic case. The first results obtained, in which the curvature was allowed to vanish, were produced by Lin. He dealt with a sufficiently smooth metric with nonnegative curvature in \cite{MR816670}, and a sufficiently smooth metric with curvature vanishing to zeroth order ($K(0)=0$ and $|\nabla K(0)| \neq 0$) in \cite{MR859276}. Many other results were obtained in the last 15 years by Han, Hong, Khuri, and Lin in \cite{MR2094852, MR2183856, MR2261749, MR2015470, MR3070566, MR2330415, MR2308865, MR2641422, MR2833424}. These papers treat much more general situations regarding the vanishing properties of $K$, and involve varying degrees of regularity for the metric and solution . For results on the higher dimensional local isometric embedding problem, see \cite{MR0726313, CCWSY, MR2900544, MR0992597, MR2754069}.

A recent result of Han and Khuri \cite{MR2765727} shows that if $K$ vanishes to finite order and its zero set consists of two Lipschitz curves intersecting transversely at a point, then a sufficiently smooth local isometric embedding exists. The next most natural generalization would be the cases when the zero set consists of two sufficiently smooth curves tangent at a point. One case that we deal with in the current paper may be represented by {Figure \ref{fig:1}}.

\begin{figure}
\centering
\includegraphics[scale=1]{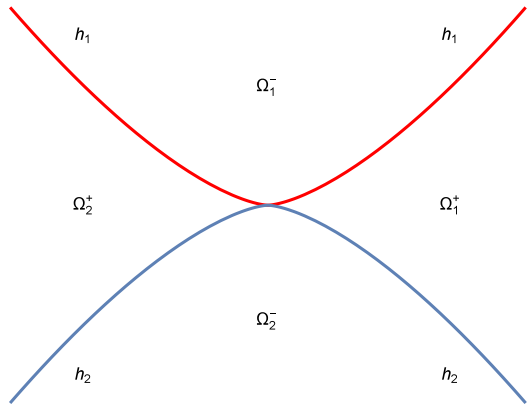}
\caption{Elliptic Cusp.
}
\label{fig:1}
\end{figure}

Without loss of generality we may assume that there are local coordinates $(\hat{x},\hat{y})$ centered at the point of intersection so that the two curves are locally given by graphs $\hat{y}=h_{i}(\hat{x})$, $i=1,2$, with $h_{1}$ and $h_{2}$ lying above and below the $\hat{x}$-axis, respectively.
In order for the regions to have reasonable cusps we require that
\begin{equation} \label{a0}
\lim_{\hat{x} \rightarrow 0} h_{i}(\hat{x}) = \lim_{\hat{x} \rightarrow 0}h'_{i}(\hat{x}) = 0,
\end{equation}
where the upper prime indicates differentiation. Moreover, standard cusps satisfy certain growth restrictions (see for example \cite{MR1643072})
\begin{eqnarray}\label{a3}
&&\frac{h_1(\hat{x})}{\hat{x}} \  \text{ is nondecreasing \ \ and \ \ } \frac{h_2(\hat{x})}{\hat{x}} \  \text{ is nonincreasing,}
\end{eqnarray}
and
\begin{equation}\label{a0.5}
\Big| h^{(l)}_{i} \cdot h^{l-1}_{i}\Big| \le C_l ,\quad\quad \text{for} \quad\quad 0\ \le\ l\ \le\ m_1,
\end{equation}
where $(l)$ denotes differentiation and $l-1$ is an exponent. In each of the four regions cut out by the curves, the Gaussian curvature has a constant sign, which is positive in $\Omega^+_{i}$, and negative in $\Omega^-_{i}$. In this family of problems, $K$ is positive in cusp domains and negative in sufficiently smooth domains.

Another family of cases that we consider may be represented by Figure \ref{fig:2}.
An obvious difference from the previous case is that $K$ is negative in cusp domains, and is positive in sufficiently smooth domains.
In both cases we will open up the cusp domain to facilitate analysis, and for this it is convenient to express each curve as a graph over the tangent line. However, in this hyperbolic cusp setting, we will need more detailed information concerning the character of the cusp. More precisely, it will be required that the graph functions satisfy
\begin{equation}\label{cusp parameter}
|h_{i}(\hat{x})|\geq C \hat{x}^{1+\bar{\alpha}}, 
\end{equation}
for appropriate ranges of $\bar{\alpha}>0$. Note that larger values of $\bar{\alpha}$ correspond to sharper cusps.

\begin{figure}
\centering
\includegraphics[scale=1]{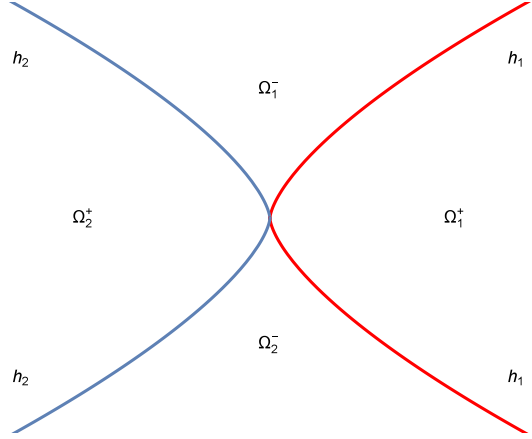}
\caption{Hyperbolic Cusp.
}
\label{fig:2}
\end{figure}

Recall that the local isometric embedding problem is equivalent to the local solvability of a Monge-Amp\'ere equation. To see this, we search for a function $z$ with $|\nabla z(0)|=0$, such that $g-dz^2$ is flat in a neighborhood of the origin. Since this metric is then locally isometric to Euclidean space, there exist functions $(u,v)$ with $g-dz^2 =du^2 +dv^2$. The map $(u,v,z)$ then provides the desired embedding. Furthermore, a direct computation shows that $g-dz^2$  is flat if and only if $z$ satisfies the so called Darboux equation
\begin{equation} \label{a1}
\det \mathrm{Hess}_g z = K (\det g)(1- |\nabla_{g} z|^2).
\end{equation}

Our strategy, following Han and Khuri \cite{MR2765727}, will be to separately analyze each region where the Gaussian curvature has a constant sign. Within such a region, we will solve a linearized version of the equation (\ref{a1}) with special boundary conditions along the zero set, as well as obtain the appropriate a priori estimates. These estimates will then allow an application of the Nash-Moser iteration to obtain a solution of the nonlinear equation in each region. Finally, we patch the solutions found in each region to form a complete solution in a whole neighborhood of the origin. The boundary conditions imposed on the linearized problems ensures that the solutions may be patched together in a smooth way across the zero set. Our main theorem is as follows.

\begin{theorem} \label{main}
Let $g \in C^{m_0}$ be a Riemannian metric defined on a neighborhood of a point in the plane, with Gaussian curvature $K$ vanishing there to the finite order $N$. Moreover let the zero set $K^{-1}(0)$ consist of two $C^{m_0}$ curves which are tangent at the point, and satisfy (\ref{a0}), (\ref{a3}) and (\ref{a0.5}).
\begin{enumerate}
\item If $K$ is positive between the two curves and negative in the complement with $m_0 \ge 3N+19$ and $m_1 \ge 2N+6$, then $g$ admits a $C^{m}$ local isometric embedding into $\mathbb{R}^3$ for all $m \le \min\{ m_1-1, \ m_0 - N - 15 \}$.
\item If $K$ is negative between the two curves and positive in the complement, 
     \eqref{cusp parameter} is satisfied with $\bar{\alpha} \in(0, 1)$, $m_0 \ge \frac{1+\alpha}{1-2\alpha}(6N+33)+4N+33$ where $\alpha = \frac{\bar{\alpha}}{1+\bar{\alpha}}$, and $m_1 \ge \frac{6N+33}{1-2 \alpha}$, then $g$ admits a $C^{m}$ local isometric embedding into $\mathbb{R}^3$ for all
     $m \le \min \left\{ m_1-1, \  \frac{m_0-(1+\alpha)(5N+34)}{(1+\alpha)^2} \right\}$.
\end{enumerate}
\end{theorem}

Note that there is no hypothesis placed on the type of cusp in case $\it{(1)}$. In $\it{(2)}$ the restriction on $\bar{\alpha}$, which will be explained more below, comes from the Nash-Moser iteration as opposed to difficulties incurred when solving the linearized problem. In fact, we have found that it is possible to invert the linearized operator without any restriction on $\bar{\alpha}$. This yields existence results for linear degenerate hyperbolic equations with a cusp Cauchy surface, which may be of independent interest. To the best of our knowledge, even this linear problem has not been previously treated in the literature. We also point out that if $-1 < \bar{\alpha} < 0 $ in $\it{(2)}$ then this is a special case of $\it{(1)}$, and if $\bar{\alpha} =0 $ then this is a special case of the results in \cite{MR2765727}.


The need to apply the Nash-Moser iteration emerges from the degeneracy of the Darboux equation when $K$ vanishes. In particular, when solving the linearized equation a certain ``loss of derivatives" occurs, which the Nash-Moser iteration is specifically designed to deal with (see \cite{MR0206461, MR0199523}). In our situation, a further difficulty arises. Namely, since fundamental tools such as the trace operator and the extension operator are not bounded in standard Sobolev spaces for cusp domains, all the energy estimates must be performed in weighted Sobolev spaces. However when using these spaces, in addition to the loss of derivatives, the loss of weighted functions also appears.

When $K$ is positive in the cusp domain, we manage to absorb the loss of weighted functions into the loss of derivatives by \cororef{c3.1}, and make the observation that the former causes no more damage than the latter. Therefore the standard Nash-Moser iteration works in this case. On the other hand, when $K$ is negative in a cusp domain, although we can still absorb the loss of the weighted functions into the loss of the derivatives, the former grows much faster than the latter when estimating higher derivatives. Since \cite{MR656198} and \cite{MR546504} provide restrictions on such growth for a successful application of the Nash-Moser iteration, we need to impose a condition on the zero set in order to stay within the appropriate framework. This condition is precisely the restriction on the values of $\bar{\alpha}$ that appears in case $(2)$ in Theorem 1.1. Such a hypothesis prevents the cusp from becoming too sharp near the origin. In other words, it prevents the initial Cauchy surface from becoming time-like too fast.

This paper is organized as follows. In Section 2, we introduce notation and review a canonical form of the linearized equation derived in \cite{MR2765727}. In Sections 3 and 5, we solve the linearized equation and obtain Moser estimates in a cusp domain where $K$ is positive, respectively negative. Sections 4 and 6 are then dedicated to the Nash-Moser iteration in the cusp domains having positive and negative curvature respectively.


\section{The linearized canonical form}
Consider here a single cusp domain $\hat{\Omega}$ associated with positive curvature. That is, it arises as in the previous section as the region between two curves which are tangent at a point, and bounds a domain with $K>0$. We may choose local coordinates $(\hat{x}, \hat{y})$ so that the boundary $\partial\hat{\Omega}$ consists of graphs $\hat{y}= \pm h( \hat{x} )$. The cusp domain itself is then given by
$$
\hat{\Omega} = \{ \ (\hat{x} , \hat{y}) \ |  \ \ \ |\hat{y}| < h(\hat{x}) , \  0 < \hat{x} < 1 \ \}.
$$
Here we are tacitly assuming without loss of generality that the domain is symmetric with respect to the axis. The function $h$ will show up frequently as a weight function or as a parameter in a coordinate change, and if the boundary is not symmetric across the axis, then the same methods still hold after taking the maximum of the two sides instead.

We will follow the set up in \cite{MR2765727}. Namely, after expressing equation \eqref{a1} in the coordinates $(\hat{x},\hat{y})$, an approximate solution $z_0$ may be constructed by using Taylor's theorem, so that
$$
\det \mathrm{Hess}_g z_0-K(\det g)(1-|\nabla_g z_0|^2)
$$
vanishes to high order at the origin. A full solution of \eqref{a1} will then be sought in the form
$z=z_0 + \varepsilon^5 w, $ where $\varepsilon > 0 $ is a small parameter. Let
$$
\Phi(w):=\det \mathrm{Hess}_g z-K(\det g)(1-|\nabla_g z|^2),
$$
and
\begin{equation*}
\begin{pmatrix}
a^{11}&a^{12}\\
a^{12}&a^{22}
\end{pmatrix}\ =\
\begin{pmatrix}
\nabla_{22}z&-\nabla_{12}z\\
-\nabla_{12}z&\nabla_{11}z
\end{pmatrix},
\end{equation*}\\
then the choice of $z_{0}$ may be arranged so that
$$\lim_{|(\hat{x},\hat{y})|\rightarrow 0}a^{22}=1,\ \ \ \ \ \  a^{12}=O\Big(|(\hat{x},\hat{y})|\Big).$$\\

Next, rescale the coordinates by
$$(\hat{x}, \hat{y})=(\varepsilon^2 x, \varepsilon^2 y),$$
so that the linearized equation
$$\mathcal{L}(w)u:=\frac{d}{dt}\Phi(w+tu)\Big|_{t=0}$$\\
takes the form
\begin{equation}
\varepsilon^{-1}\mathcal{L}(w)u=a^{ij} \nabla_{ij}u+2\varepsilon^4K|g|\langle \nabla_g z, \nabla_g u \rangle, \label{a4}
\end{equation}
where $|g| = \text{det} g$ and $\nabla_{ij}$ denote covariant derivatives with respect to the rescaled coordinates.

In order to facilitate analysis of the linearized equation, it is helpful to eliminate the second order mixed derivative term. This may be accomplished by a further coordinate change to $(\bar{x},\bar{y})$, such that $\bar{y}=y$ and
\begin{equation*}
a^{12} \frac{\partial \bar{x}}{\partial x} + a^{22} \frac{\partial \bar{x}}{\partial y}=0.
\end{equation*}
This equation may be solved with the method of characteristics while imposing the following initial condition
\begin{equation*}
\bar{x}( x, 0) = x.
\end{equation*}
When expressed in these new coordinates, the linearized equation takes the so called canonical form that played an important role \cite{MR2765727}. We summarize the result in the following lemma.

\begin{lemma} \label{l1}
Let g $\in C^{m_0}$ and $w \in C^{\infty}$ with $|w|_{C^4} \le 1$. In a neighborhood of the origin there exists a local $C^{m_0-2}$ change of coordinates
\begin{equation*}
\bar{x}=\bar{x}(x, y), \ \ \ \ \ \  \bar{y}=y,
\end{equation*}
such that in these new coordinate system the linearization takes the form
$$
\varepsilon^{-1} \mathcal{L}(w)u=a^{22}L(w)u+(a^{22})^{-1} \Phi(w)[ \partial^2_{x}u - \partial_{x} \mathrm{log}(a^{22} \sqrt{|g|}) \partial_{x}u],
$$
where
\begin{equation}
L(w)u= \partial_{\bar{x}} (k \partial_{\bar{x}})u + \partial^2_{\bar{y}}u+ c \partial_{\bar{x}}u + d \partial_{\bar{y}} u, \label{a6}
\end{equation}
with
\begin{eqnarray*}
k &=& K \bar{k}(x,y,w, \nabla w,\nabla ^2 w, \nabla \bar{x}),\\
c &=& K \bar{c}(x,y,w, \nabla w,\nabla ^2 w, \nabla ^3 w, \nabla \bar{x}, \nabla ^2 \bar{x})+(a^{22})^{-2} \partial_{x} \Phi (w) \partial_{x} \bar{x},\\
d &=& \varepsilon^2 \bar{d}(x,y,w, \nabla w,\nabla ^2 w),
\end{eqnarray*}
for some $\bar{k}, \bar{c}, \bar{d} \in C^{m_0-4}$ such that $\bar{k}>1/2$ if $\varepsilon=\varepsilon(m)$ is chosen sufficiently small. Moreover, we have an estimate on the new coordinates
\begin{equation}
\| (\bar{x},\bar{y}) \|_{H^m(\Omega)} \le C_m ( 1+ \| w \| _{H^{m+4}(\Omega)}), \label{a7}
\end{equation}
with the constant $C_m$ independent of $\varepsilon$.
\end{lemma}

We remark that the intersection of the original domain $\hat{\Omega}$ with a small neighborhood of the origin may be written in the new coordinates as
$$
\bar{\Omega}= \{ \ (\bar{x},\bar{y})\ | \ \ \ |\bar{y}| <\bar{h}(\bar{x}) \ , \  0 <\bar{x} < 1 \ \}.
$$
Here $\bar{h}$ is a rescaling of the initial graph $h$ by $\varepsilon^{-2}$. It is important to note that even with this rescaling of the graph, the condition (\ref{a0.5}) still holds in the new coordinates with a constant that is independent of $\varepsilon$. Since this is the primary hypothesis needed to treat the elliptic regions, and it is not affected by the rescaled graph, we will in what follows continue to denote the graph by $h$ instead of $\bar{h}$ for simplicity.

In conclusion we have defined three different coordinate systems which culminated in a rescaling of the domain, and a canonical form for the linearization. In all of these coordinates the geometry of the domain remains a cusp. However, in the next section, in order to facilitate analysis of the linearized equation, we will open up the cusp by another set of coordinates $(\tilde{x}, \tilde{y})$, in which the domain is transformed into an infinite cylinder denoted by $\tilde{\Omega}$.


\section{Linear theory in elliptic cusps}\label{s2}
According to our assumption that $K>0$ in $\bar{\Omega}$ and vanishes along $\partial \bar{\Omega}$, the linearized equation (\ref{a6}) is of degenerate elliptic type. We need to establish existence and to prove appropriate energy estimates. The framework of the computation is from \cite{MR2765727}, where the authors dealt with a sector domain. It is well known that in a sector domain, the regularity of elliptic equations depends on the angle of the sector. Smaller angles yield higher regularity. Since in \cite{MR2765727} angles are not necessarily small, the authors needed to use the degeneracy to overcome this problem. Here we will use an extra coordinate change and exploit degeneracy more directly to derive an easier computation. However it turns out that the use of degeneracy is not necessary in the case of cusp domains, because the angle of a cusp zero. Nevertheless we will use the degeneracy anyway, but point out how it can be replaced by the zero angle in the computation.

First, we make a coordinate change by
\begin{equation*}
\tilde{x}=\int_{\bar{x}}^1 \frac{1}{h(t)}dt \text{,} \ \ \ \ \ \ \ \tilde{y}=\frac{\bar{y}}{h(\bar{x})},
\end{equation*}
to bring $\bar{\Omega}$ to
\begin{equation*}
\tilde{\Omega} = \{ \ (\tilde{x},\tilde{y}) \ | \ \ \ \tilde{x} \ge 0, \ -1 \le \tilde{y} \le 1 \ \}.
\end{equation*}
The Jacobian of the coordinate change is\\
\begin{equation*}
\begin{pmatrix}
\frac{\partial \tilde{x}}{\partial \bar{x}}&\frac{\partial \tilde{x}}{\partial \bar{y}}\\
\frac{\partial \tilde{y}}{\partial \bar{x}}&\frac{\partial \tilde{y}}{\partial \bar{y}}
\end{pmatrix}\ =\
\begin{pmatrix}
\frac{-1}{h}&0\\
\frac{-\tilde{y}h'}{h}&\frac{1}{h}
\end{pmatrix},
\end{equation*}
\begin{equation*}
\begin{pmatrix}
\frac{\partial \bar{x}}{\partial \tilde{x}}&\frac{\partial \bar{x}}{\partial \tilde{y}}\\
\frac{\partial \bar{y}}{\partial \tilde{x}}&\frac{\partial \bar{y}}{\partial \tilde{y}}
\end{pmatrix}\ =\
\begin{pmatrix}
-h&0\\
-\tilde{y}h'&h
\end{pmatrix},
\end{equation*}
and the equation (\ref{a6}) becomes
\begin{equation} \label{a8}
L(w)u = \partial_{\tilde{x}}(\mathcal{K} \partial_{\tilde{x}} u) + \partial_{\tilde{x}}(\mathcal{A} \partial_{\tilde{y}} u) + \partial_{\tilde{y}}(\mathcal{A} \partial_{\tilde{x}} u) + \partial_{\tilde{y}}(\mathcal{B} \partial_{\tilde{y}} u) + \mathcal{C} \partial_{\tilde{x}} u + \mathcal{D} \partial_{\tilde{y}} u,
\end{equation}
with
\begin{eqnarray*}
\mathcal{K} &=& \frac{k}{h^2}, \\
\mathcal{A} &=& \frac{k\tilde{y} h'}{h^2},\\
\mathcal{B} &=& \frac{1}{h^2}+k(\frac{\tilde{y} h'}{h})^2,\\
\mathcal{C} &=& \frac{-c}{h}-\frac{2k h'}{h^2},\\
\mathcal{D} &=& \frac{d}{h}-\frac{c\tilde{y} h'}{h^2}-\frac{2k\tilde{y} h'}{h^2}.
\end{eqnarray*}

We remark here that $h= h(\bar{x}(\tilde{x}))$. An important consequence of this is:
\begin{lemma} \label{l0}
For all nonnegative integer $m$ and real number $\tau$, we have
$$\frac{d^m h^{\tau}}{d\tilde{x}^m} = O(h^{\tau}).$$
\end{lemma}
\begin{proof}
Notice that
$$
\frac{dh(\bar{x})}{d\tilde{x}}=h' \frac{d\bar{x}}{d\tilde{x}},
$$
and
$$
\frac{\partial \bar{x}}{\partial \tilde{x}}=-h(\bar{x}),
$$
so by (\ref{a0.5}) the inequality is true for $m=1$ and $\tau=1$. A simple induction yields the desired result.
\end{proof}
 We also remark that we merely used the boundedness of (\ref{a0.5}) for the proof above. In a cusp domain we actually have
$$
\Big|h^{(l)}\cdot h^{l-1}\Big| \rightarrow 0 \ \ \ \ \text{as} \  \varepsilon \rightarrow 0, \ \ \ \text{ for } l\ge1,
$$
and therefore
\begin{equation} \label{a9}
\Big|\frac{d^l h^{\tau}(\bar{x})}{d\tilde{x}^l} h^{-\tau}(\bar{x}) \Big| \rightarrow 0 \ \ \ \ \text{as} \ \varepsilon \rightarrow 0,\ \ \ \text{ for }l\ge1.
\end{equation}
This can be used to deal with non-degenerate cases.

It will be convenient to cut off the coefficients away from the infinity of the $(\tilde{x}, \tilde{y})$ plane. Let $\varphi \in C^{\infty}([0,\infty))$ be a nonnegative cut-off functioin with
\begin{equation*}
\varphi(\tilde{x})=
\begin{cases}
1, \ & \text{if } \tilde{x} \ge 2,\\
0, & \text{if } \tilde{x} \le 1,
\end{cases}
\end{equation*}
and define
\begin{equation} \label{a10}
Lu= \partial_{\tilde{x}}(\bar{K} \partial_{\tilde{x}} u) + \partial_{\tilde{x}}(\bar{A} \partial_{\tilde{y}} u) + \partial_{\tilde{y}}(\bar{A} \partial_{\tilde{x}} u) + \partial_{\tilde{y}}(\bar{B} \partial_{\tilde{y}} u) + \bar{C} \partial_{\tilde{x}} u + \bar{D} \partial_{\tilde{y}} u,
\end{equation}
where
\begin{equation*}
\bar{K}=\varphi^2 \mathcal{K},\  \bar{A}=\varphi \mathcal{A}, \  \bar{B} = \mathcal{B}, \  \bar{C}=\varphi \mathcal{C}, \  \bar{D}=\varphi \mathcal{D}.
\end{equation*}
We will study the boundary value problem:
\begin{equation}\label{a11}
 Lu=f \text{ in } \tilde{\Omega}, \ \ \ \ \ \ u(\tilde{x},1)=u(\tilde{x},-1)=0.
\end{equation}
Since there is ${1}/{h}$ in the coefficient of $L$, which blows up when $\tilde{x}$ approaches infinity, we need the solution to vanish at the infinity. It will be convenient to define the following weighted Sobolev norm to control the rate at which the functions vanish:

\begin{equation*}
\|u\|^2_{(m,l,\gamma)}=\sum_{\substack{s \le m, \ t \le l \\ t+s \le \text{max}(m,l)}} \int_{\tilde{\Omega}} \lambda^{-s} |h^{-\gamma}\ \partial_{\tilde{x}}^s \partial^t_{\tilde{y}} u|^2,
\end{equation*}
where $\lambda$ is a large parameter. Let $\overline{C}^{\infty}(\tilde{\Omega})$ be the space of the smooth functions that vanish for $\tilde{x}$ large and $H^{(m,l,\gamma)}(\tilde{\Omega})$ be its closure with respect to the norm defined above. For all the weighted norms in this paper, if $m=l$ we will drop the $l$. We also let $\hat{C}^{\infty}(\tilde{\Omega})$ denote the space of $\overline{C}^{\infty}(\tilde{\Omega})$ functions $v$ satisfying $v(\tilde{x},1)=v(\tilde{x},-1)=0$. For every $f\in H^{(m,l,\gamma)}(\tilde{\Omega})$, we search for a weak solution $u\in H^{(m,l,\gamma)}(\tilde{\Omega})$ satisfying
\begin{equation} \label{a11.5}
(u,L^*v)=(f,v) \quad v\in \hat{C}^{\infty}(\tilde{\Omega}),
\end{equation}
where $(\cdot,\cdot)$ is the $L^2(\tilde{\Omega})$ inner product and $L^*$ is the formal adjoint of $L$. We first solve an auxiliary ODE.
\begin{lemma} \label{l2}
For every $v\in \hat{C}^{\infty}$, there exists a unique solution $\xi \in H^{(2m,\infty,\gamma)}(\tilde{\Omega}) \cap C^{\infty}(\tilde{\Omega})$ of the ODE:\\
\begin{eqnarray}
&&\sum_{s=0}^m \lambda^{-s} \partial^s_{\tilde{x}}( \frac{\tilde{y}^2-2}{h^{2\gamma}} \partial_{\tilde{x}}^s \xi) = v,\\
&&\xi(\tilde{x},-1)=\xi(\tilde{x},1)=0, \quad \partial_{\tilde{x}}^s \xi(0,\tilde{y})=0, \quad 0\le s\le m-1, \nonumber \\
&&h^{-2\gamma}(\bar{x}_0)\int_{\tilde{x}=\tilde{x}_0}(\partial_{\tilde{x}}^s \partial_{\tilde{y}}^t \xi)^2 < \infty, \quad 0 \le s \le 2m-1, \quad t < \infty. \nonumber
\end{eqnarray}
\end{lemma}
\begin{proof}
Denote the completion, with respect to $\| \cdot \|_{(m,0,\gamma)}$, of $C^{\infty}(\tilde{\Omega})$ functions with compact support by $H_0^{(m,0,\gamma)}(\tilde{\Omega})$. We first obtain a weak solution in $H_0^{(m,0,\gamma)}(\tilde{\Omega})$ by the Riesz representation theorem, so the desired boundary behavior at $\{\tilde{y}=\pm 1\}$ and $\{\tilde{x}=0\}$ follows.
Then we define
$$
\rho_{\tilde{x}_0}(\tilde{x})=
\begin{cases}
1, & \text{if } \tilde{x}\le \tilde{x}_0,\\
0, & \text{if } \tilde{x}\ge \tilde{x}_0+1.
\end{cases}
$$
and integrate the following by parts:
\begin{equation*}
\int_{\tilde{\Omega}_{3\tilde{x}_0}}\Big( v- \sum_{s=0}^m(-1)^s \lambda^{(-s)}\partial^s_{\tilde{x}} (\frac{\tilde{y}^2-2}{h^{2\gamma}} \partial^s_{\tilde{x}} \xi)\Big) \rho^{2m}_{\tilde{x}_0} \partial_{\tilde{x}}^{2 \tau} \xi = 0,
\end{equation*}
where $\tilde{\Omega}_{\tilde{x}_0}$ is defined to be $\{(\tilde{x},\tilde{y})\ |\ 0 \le \tilde{x} \le \tilde{x}_0, -1 \le \tilde{y} \le 1 \}$. For each $1 \le \tau \le m$, we integrate by parts so that the highest order derivative on $\xi$ is $m+ \tau$. We keep only the highest derivative on the left hand side of the equality and the remaining terms on the right hand side. Notice that the boundary integrals are either zero or independent of $\tilde{x}_0$, and the derivatives of $\rho_{\tilde{x}_0}$ is bounded by a constant independent of $\tilde{x}_0$. Therefore, by taking limit as $\tilde{x}_0$ goes to infinity and applying a simple induction on $\tau$, we arrive at $\xi \in H^{(2m,0,\gamma)}(\tilde{\Omega})$. This is actually the standard regularity estimate for elliptic equations. Differentiating the equation with respect to $\tilde{y}$ and repeating the argument above yields $\xi \in H^{(2m,l,\gamma)}(\tilde{\Omega})$ for arbitrary $l \ge 1$. $\xi \in C^{\infty}(\tilde{\Omega})$ is a corollary of the Sobolev embedding theorem and regularity of ODE.

Lastly, consider $h^{-\gamma} \partial_{\tilde{x}}^s \partial_{\tilde{y}}^t \xi \in H^m(\tilde{\Omega})$, where $H^m(\tilde{\Omega})$ is the usual Sobolev space. The vanishing of the integral along $\tilde{x}=\tilde{x}_0$ follows directly from the trace theorem in the Sobolev space.
\end{proof}

The next lemma is the main tool for all the energy estimates in this section.
\begin{lemma} \label{l3}
Suppose that $|w|_{C^4(\Omega)}<1$ and let $u\in C^2(\tilde{\Omega})$ with $u(\tilde{x},1)=u(\tilde{x},-1)=0$. If $\varepsilon$ is sufficiently small and
\begin{equation} \label{a12}
h^{-2\gamma-2}(\bar{x}_0)\int_{\tilde{x}=\tilde{x}_0} (\partial_{\tilde{x}}^su)^2 <C <\infty,\quad  \text{for } s=0,1.
\end{equation}
Then
$$
\int_{\tilde{\Omega}}Lu \frac{\tilde{y}^2-2}{h^{2\gamma}}u \ge  \int_{\tilde{\Omega}} \frac{(\partial_{\tilde{y}}u)^2}{h^{2\gamma+2}}+\frac{u^2}{h^{2\gamma+2}}.
$$
\end{lemma}
\begin{proof}
For any functions $a\in C^{\infty}(\tilde{\Omega})$, integrating by part yields
\begin{eqnarray*}
\int_{\tilde{\Omega}_{\tilde{x}_0}} au Lu &=& - \int_{\tilde{\Omega}{\tilde{x}_0}}a[\bar{K}(\partial_{\tilde{x}}u)^2+2\bar{A}\partial_{\tilde{x}}u \partial_{\tilde{y}}u +\bar{B}(\partial_{\tilde{y}} u)^2]\\
&&+\int_{\tilde{\Omega}_{\tilde{x}_0}} \frac{u^2}{2} [\partial_{\tilde{x}} (\bar{K} \partial_{\tilde{x}}a)+ \partial_{\tilde{x}}( \bar{A} \partial_{\tilde{y}} a) + \partial_{\tilde{x}\tilde{y}}(\bar{A}a)\\
&& \quad \quad \quad \quad -\partial_{\tilde{y}}(\partial_{\tilde{x}} \bar{A}a) + \partial_{\tilde{y}}(\bar{B}\partial_{\tilde{y}}a)- \partial_{\tilde{x}}( \bar{C}a)- \partial_{\tilde{y}}( \bar{D}a)]\\
&&-\int_{\partial \tilde{\Omega}_{\tilde{x}_0}} \frac{u^2}{2} [\bar{K} \partial_{\tilde{x}}a \nu^1 +\bar{A} \partial_{\tilde{y}} a \nu^1 + \partial_{\tilde{y}}( \bar{A}a) \nu^1\\
&& \quad \quad \quad \quad \quad -\partial_{\tilde{x}} \bar{A}a \nu^2 +\bar{B} \partial_{\tilde{y}}a \nu^2 -\bar{C}a \nu^1- \bar{D}a \nu^2]\\
&& +\int_{\partial \tilde{\Omega}_{\tilde{x}_0}} a[\bar{K}\partial_{\tilde{x}}uu \nu^1+2 \bar{A} \partial_{\tilde{x}}uu \nu^2+ \bar{B}\partial_{\tilde{y}}uu \nu^2].
\end{eqnarray*}
Letting $a= \frac{\tilde{y}^2-2}{h^{2\gamma}}$, then the first integral becomes:
\begin{eqnarray*}
&&-\int_{\tilde{\Omega}_{\tilde{x}_0}} \frac{\tilde{y}^2-2}{h^{2\gamma+2}}k[\varphi^2 (\partial_{\tilde{x}}u)^2+2\varphi \tilde{y} h' \partial_{\tilde{x}}u \partial_{\tilde{y}}u +(\tilde{y}h')^2 (\partial_{\tilde{y}}u)^2]+ \frac{\tilde{y}^2-2}{h^{2\gamma+2}}(\partial_{\tilde{y}}^2 u)^2\\
&= &-\int_{\tilde{\Omega}_{\tilde{x}_0}} \frac{\tilde{y}^2-2}{h^{2\gamma+2}}k[ \varphi (\partial_{\tilde{x}} u) + \tilde{y}h'(\partial_{\tilde{y}} u)]^2-\int_{\tilde{\Omega}_{\tilde{x}_0}} \frac{\tilde{y}^2-2}{h^{2\gamma+2}}(\partial_{\tilde{y}}u)^2\\
&\ge & \int_{\tilde{\Omega}_{\tilde{x}_0}}\frac{(\partial_{\tilde{y}}u)^2}{h^{2\gamma+2}}.
\end{eqnarray*}
By the degeneracy of $k$ and the fact that $D=O(\varepsilon^2)$ we have
$$
|\partial_{\tilde{x}} (\bar{K} \partial_{\tilde{x}}a)|+ |\partial_{\tilde{x}}( \bar{A} \partial_{\tilde{y}} a)| + |\partial_{\tilde{x}\tilde{y}}(\bar{A}a)|-|\partial_{\tilde{y}}(\partial_{\tilde{x}} \bar{A}a)| - |\partial_{\tilde{x}}( \bar{C}a)|- |\partial_{\tilde{y}}( \bar{D}a)| \le o(\varepsilon) h^{-2\gamma-2},
$$
and
$$\quad \partial_{\tilde{y}}(\bar{B} \partial_{\tilde{y}} a) \ge (2-o(\varepsilon))h^{-2\gamma-2}.
$$
Let $\tilde{x}_0 \rightarrow \infty$, the desired result follows because the boundary integral vanishes by the degeneracy of $k$ and the fact that $h$ approaches zero when $\tilde{x}_0 \rightarrow \infty$.
\end{proof}

We remark here that \lemmaref{l3} is the main part where we use the degeneracy. First of all, degeneracy helps positive terms to dominate. In a cusp domain, we may instead use (\ref{a9}) to achieve the same consequence. Secondly, the degeneracy also assures that the boundary integral vanishes. Without the degeneracy, we may demand a little stronger on the condition (\ref{a12}) that the integral on the left hand side goes to zero when $\tilde{x}_0 \rightarrow \infty$, not just being bounded. Although this modification on (\ref{a12}) requires some minor changes on the derivation of the existence below, we point out that the degeneracy is not necessary in the cusp domain. Now we are ready to obtain existence.

\begin{theorem} \label{t1}
Suppose that $g\in C^{m_0}$, $|w|_{C^4(\Omega)}<1$ and $f\in H^{(m,1,\gamma)}(\tilde{\Omega})$, with the condition $\partial_{\tilde{x}}^s f(0,\tilde{y})=0 $ for $0 \le s \le m-1$. If $m \le m_0-4$ and $\varepsilon=\varepsilon(m)$ sufficiently small, then there exists a unique weak solution $u \in H^{(m,1,\gamma)}(\tilde{\Omega})$ of (\ref{a11.5}).
\end{theorem}
\begin{proof}
Given $v \in \hat{C}^{\infty}(\tilde{\Omega})$, by \lemmaref{l2}, there exists $\xi \in H^{(m,1,\gamma+2)}(\tilde{\Omega})$ such that
\begin{eqnarray} \label{a13}
&&\sum_{s=0}^m \lambda^{-s} \partial^s_{\tilde{x}}( \frac{\tilde{y}^2-2}{h^{2\gamma+2}} \partial_{\tilde{x}}^s \xi) = v,\\
&&\xi(\tilde{x},-1)=\xi(\tilde{x},1)=0, \ \ \ \ \ \ \partial_{\tilde{x}}^s \xi(0,\tilde{y})=0, \ \ \ \ \text{for} \ \ 0\le s\le m-1, \nonumber \\
&&h^{-2(\gamma-2)}(\bar{x}_0)\int_{\tilde{x}=\tilde{x}_0}(\partial_{\tilde{x}}^s \partial_{\tilde{y}}^t \xi)^2 < \infty, \ \ \ \text{for} \ \ 0 \le s \le 2m-1, \ \ t < \infty. \nonumber
\end{eqnarray}

Our first step is to establish an estimate
\begin{equation} \label{a14}
\Bigg(L \xi, h^2 \sum_{s=0}^m (-1)^s \lambda^{-s} \partial_{\tilde{x}}^s(\frac{\tilde{y}^2-2}{h^{2\gamma+2}}\partial_{\tilde{x}}^s \xi ) \Bigg) \ge C \| \xi \|^2_{(m,1,\gamma)}.
\end{equation}
We first integrate the left hand side of (\ref{a14}) by parts to see
\begin{eqnarray} \label{a15}
&&\Bigg(L \xi, h^2 \sum_{s=0}^m (-1)^s \lambda^{-s} \partial_{\tilde{x}}^s(\frac{\tilde{y}^2-2}{h^{2\gamma+2}}\partial_{\tilde{x}}^s \xi ) \Bigg)\\ \nonumber
&=& \sum_{s=0}^m \int_{\tilde{\Omega}} \lambda^{-s} L(\partial_{\tilde{x}}^s \xi)\frac{\tilde{y}^2-2}{h^{2\gamma}}\partial_{\tilde{x}}^s \xi +\sum_{s=1}^m \int_{\tilde{\Omega}} \lambda^{-s}[\partial_{\tilde{x}}^s, L] \xi \frac{\tilde{y}^2-2}{h^{2\gamma}}\partial_{\tilde{x}}^s \xi \nonumber \\
&& \quad +\sum_{s=1}^m \sum_{l=0}^{s-1} \int_{\tilde{\Omega}} \lambda^{-s} \binom{s}{l}\partial_{\tilde{x}}^{s-l}h^2 \partial_{\tilde{x}}^l(L \xi)\frac{\tilde{y}^2-2}{h^{2\gamma+2}}\partial_{\tilde{x}}^s \xi. \nonumber
\end{eqnarray}
The boundary behavior in (\ref{a13}) guarantees all the boundary integrals in (\ref{a15}) vanish, and provides the condition assumed in \lemmaref{l3}. So we may use \lemmaref{l3} and pick $\varepsilon$ small so that
\begin{equation}\label{a16}
\sum_{s=0}^m \int_{\tilde{\Omega}} \lambda^{-s} L(\partial_{\tilde{x}}^s \xi) \frac{\tilde{y}^2-2}{h^{2\gamma}} \partial_{\tilde{x}}^s \xi \ge \sum_{s=0}^m \lambda^{-s} \int_{\tilde{\Omega}} \frac{(\partial_{\tilde{y}} \partial_{\tilde{x}}^s \xi)^2}{h^{2\gamma+2}}+ \frac{(\partial_{\tilde{x}}^s \xi)^2}{h^{2\gamma+2}}.
\end{equation}
For the other two integrals in (\ref{a15}), we are going to absorb them into the positive terms in (\ref{a16}). Consider the integral with a commutator, we use the following to illustrate our computation
\begin{eqnarray}
&& \bigg|\sum_{s=1}^m \int_{\tilde{\Omega}} \lambda^{-s}[\partial_{\tilde{x}}^{s+1}(\bar{K} \partial_{\tilde{x}} \xi) -\partial_{\tilde{x}}(\bar{K} \partial_{\tilde{x}}^{s+1} \xi)] \frac{\tilde{y}^2-2}{h^{2\gamma}} \partial_{\tilde{x}}^s \xi \bigg| \nonumber \\
&\le & \sum_{s=1}^m \sum_{l=0}^s \bigg|  \int_{\tilde{\Omega}} \lambda^{-s}  \binom{s}{l}  (\partial_{\tilde{x}}^{l+1} \frac{\varphi^2 k}{h^2}) (\partial_{\tilde{x}}^{s-l+1}\xi) \frac{\tilde{y}^2-2}{h^{2\gamma}} \partial_{\tilde{x}}^s \xi \bigg| \nonumber \\
\label{a17}
& \le & C_m \sum_{s=1}^m \sum_{j=0}^1 \lambda^{-s} \bigg|\int_{\tilde{\Omega}} \partial_{\tilde{x}} \big[\partial_{\tilde{x}}^j(\varphi^2 k) \frac{\tilde{y}^2-2}{h^{2\gamma+2}} \big] (\partial_{\tilde{x}}^s \xi)^2 \bigg|  \\
&& \quad + C_m \sum_{s=1}^m \sum_{j=0}^2 \lambda^{-s} \bigg|\int_{\tilde{\Omega}} \partial_{\tilde{x}}^j(\varphi^2 k) \frac{\tilde{y}^2-2}{h^{2\gamma+2}} (\partial_{\tilde{x}}^s \xi)^2 \bigg| \nonumber\\
&& \quad + C_m \sum_{s=1}^m \sum_{l=2}^s \sum_{j=1}^l \lambda^{-s} \bigg| \int_{\tilde{\Omega}} \partial_{\tilde{x}}^j(\varphi^2 k) \frac{\tilde{y}^2-2}{h^{2\gamma+2}} (\partial_{\tilde{x}}^{s-l+1} \xi) \partial_{\tilde{x}}^s \xi \bigg| \nonumber.
\end{eqnarray}
The first two terms in the right hand side of the last inequality of (\ref{a17}) have $k$ differentiated no more than twice, so by the assumption $|w|_{C^4}<1$, we can then pick $\varepsilon$ small so that those two terms are bounded by
$$
\frac{1}{4} \sum_{s=1}^m \lambda^{-s} \int_{\Omega} \frac{(\partial_{\tilde{x}}^s \xi)^2}{h^{2\gamma+2}}.
$$
The other term in (\ref{a17}) has to be estimated as well, but we cannot abuse $\varepsilon$. Since these terms have higher derivatives of the coefficients, which contain higher derivatives of $w$, the Nash-Moser iteration prohibits us from controlling them with $\varepsilon$. Instead, we apply Cauchy-Schwartz inequality with a small $\delta$
\begin{eqnarray*}
&&\sum_{s=1}^m \sum_{l=2}^s \sum_{j=1}^l C_m \lambda^{-s} \bigg| \int_{\Omega} \partial_{\tilde{x}}^j(\varphi k) \frac{\tilde{y}^2-2}{h^{2\gamma+2}} (\partial_{\tilde{x}}^{s-l+1} \xi) \partial_{\tilde{x}}^s \xi \bigg| \\
&\le & \sum_{s=1}^m \sum_{l=2}^s \sum_{j=1}^l C_m \delta \lambda^{-s} \int_{\tilde{\Omega}} \frac{(\partial_{\tilde{x}}^s \xi)^2}{h^{2\gamma+2}} + \lambda^{-s} \int_{\tilde{\Omega}} \big[C_m \frac{\partial_{\tilde{x}}^j(\varphi k)}{4 \delta} \big] \frac{(\partial_{\tilde{x}}^{s-l+1} \xi)^2} {h^{2\gamma+2}}.
\end{eqnarray*}
The first term in the right hand side can be absorbed to (\ref{a16}) when $\delta$ chosen small.

We observe that the second term only has $\partial_{\tilde{x}}^{s-l+1} \xi$ with $s-l+1$ strictly less than $s$. So if we pick $\lambda$ larger than $\bigg{\|} C_m \frac{\partial_{\tilde{x}}^j(\varphi k)}{4 \delta} \bigg{\|}_{L^{\infty}(\tilde{\Omega})} $, this term can be absorbed into
$$
\lambda^{-s+l-1}\int_{\tilde{\Omega}}\frac{(\partial_{\tilde{x}}^{s-l+1} \xi)^2}{h^{2\gamma+2}}
$$
in (\ref{a16}). Treat all the other terms in a similar way, we derived (\ref{a14}).

We then work on functional analysis. We have for any $v \in \hat{C}^{\infty}(\tilde{\Omega})$
\begin{eqnarray} \label{a18}
\|h^2v \|_{(-m,-1,\gamma)}&:=& \sup_{\eta \in H^{(m,1,\gamma)}(\tilde{\Omega})} \frac{|(\eta,h^2v)|}{\| \eta \|_{(m,1,\gamma)}}\\
&=&\sup_{\eta \in H^{(m,1,\gamma)}(\tilde{\Omega})} \frac{|(h(y^2-2) \eta, \xi)_{(m,0,\gamma)}|}{\| \eta \|_{(m,1,\gamma)}} \nonumber \\
& \le & C \sup_{\eta \in H^{(m,1,\gamma)}(\tilde{\Omega})} \frac{\|\eta\|_{(m,0,\gamma)}\|\xi\|_{(m,0,\gamma)}}{\| \eta \|_{(m,1,\gamma)}} \nonumber \\
& \le & C \| \xi \|_{(m,1,\gamma)}. \nonumber
\end{eqnarray}
Here $( \cdot,\cdot )_{(m,0,\gamma)}$ denotes the inner product on $H^{(m,0,\gamma)}(\tilde{\Omega})$, and in the dual space $H^{(-m,-1,\gamma)}(\tilde{\Omega})$ of $H^{(m,1,\gamma)}(\tilde{\Omega})$ the norm  $\|\cdot\|_{(-m,-1,\gamma)}$ is induced naturally.

Now apply (\ref{a14}) to obtain
\begin{eqnarray*}
\| \xi \|_{(m,1,\gamma)}\|L^*h^2v \|_{(-m,-1,\gamma)} &\ge & ( \xi, L^*h^2v)\\
&=& (L \xi, h^2v)\\
&=&\bigg( L \xi, h^2 \sum_{s=0}^m \lambda^{-s}(-1)^s \partial_{\tilde{x}}^s (\frac{\tilde{y}^2-2}{h^{2\gamma+2}} \partial_{\tilde{x}}^s \xi) \bigg)\\
&\ge & C\| \xi \|^2_{(m,1,\gamma)},
\end{eqnarray*}
together with (\ref{a18}) we arrive at
\begin{equation} \label{a19}
\|h^2 v\|_{(-m,-1,\gamma)} \le C \|L^*h^2v \|_{(-m,-1,\gamma)}.
\end{equation}
Define the linear functional $F:L^*(h^2 \hat{C}^{\infty}(\tilde{\Omega})) \rightarrow \mathbb{R}$ by
$$
F(L^* h^2 v)=(f, h^2v).
$$
Then the following show that $F$ is bounded on the subspace $L^*(h^2 \hat{C}^{\infty}(\tilde{\Omega})) $ of $H^{(-m,-1,\gamma)}(\tilde{\Omega})$,
\begin{eqnarray*}
|F(L^*h^2v)|=|(f,h^2v)| &\le & \| f \|_{(m,1,\gamma)}\|h^2v\|_{(-m,-1,\gamma)}\\
& \le & C \| f \|_{(m,1,\gamma)} \|L^* h^2 v \|_{(-m,-1,\gamma)}.
\end{eqnarray*}
Notice that here we use our hypothesis that $f\in H^{(m,1,\gamma)}(\tilde{\Omega})$ and $\partial_{\tilde{x}}^s f(0,y)=0 $ for $0 \le s \le m-1$. We can then apply the Hahn-Banach theorem to obtain a bounded extension of $F$ (still denoted $F$) defined on all of $H^{(-m,-1,\gamma)}(\tilde{\Omega})$. Then by the Riesz representation theorem there exists a weak solution $u\in H^{(m,1,\gamma)}(\tilde{\Omega})$ of (\ref{a11.5}).
\end{proof}

We then boot-strap the regularity and derive an \textit{a priori} estimate needed for the Nash-Moser iteration. We introduce here a weighted Sobolev norm associated with the cusp domains $\Omega$,

\begin{equation*}
\| u \|^2_{(m,l,\gamma,\Omega)}= \int_{\Omega} \sum_{\substack{ s \le m, t \le l\\ t+s \le \max(m,l)}} | h^{(-\gamma+t+s)} \partial_{x}^s  \partial_{y}^t u |^2.
\end{equation*}

Let $\overline{C}^{\infty}(\Omega)$ be the collection of $C^{\infty}(\Omega)$ functions vanishing in a neighborhood of the origin. $H^{(m,l,\gamma)}(\Omega)$ is defined to be the closure of $\overline{C}^{\infty}(\Omega)$ with respect to $\|\cdot \|_{(m,l,\gamma,\Omega)}$. Notice that lower derivatives of $H^{(m,l,\gamma)}(\Omega)$ functions vanish faster than those of higher derivatives, but all the derivatives of $H^{(m,l,\gamma)}(\tilde{\Omega})$ functions defined before vanish at the same speed. However, a simple coordinate change shows $\|\cdot \|_{(m,l,\gamma,\Omega)}$ is equivalent to $\| \cdot \|_{(m,l,\gamma,\tilde{\Omega})}$. The notation $\mathbb{R}^2_+$ below denote the right half plane in $(x,y)$ coordinate, $\| \cdot \|_{H^m(\mathbb{R}^2_+)}$ and $H^m(\mathbb{R}^2_+)$ are usual Sobolev norm and spaces.

\begin{theorem} \label{t2}
Suppose that $g\in C^{m_0} $ and $f \in \overline{C}^{\infty}(\Omega)$. If  $|w|_{C^6(\Omega)} \le 1$ and $\varepsilon=\varepsilon(m)$ is sufficiently small then there exists a unique solution $u \in H^{(m,\gamma)}(\Omega) \cap  C^{m_0-3 }(\Omega)$ of (\ref{a11}). Moreover, there exists a constant $C_m$ independent of $\varepsilon$ such that

$$
\|u\|_{H^m(R^2_+)} \le C_m \left( \|w\|_{H^{m+6}(R^2_+)}\|f\|_{H^{2}(\mathbb{R}^2_+)}+\|f\|_{H^{m}(\mathbb{R}^2_+)}
 \right).
$$

\end{theorem}
\begin{proof}
 Since $L$ is strictly elliptic in every compact subset inside $\tilde{\Omega}$ and the coefficients of $L$ as well as $f$ are at least $C^{m_0-4}$, we obtain that $u \in H^{(m,1,\gamma)}(\tilde{\Omega}) \cap C^{m_0-3}$. The fact that $u \in H^{(m,\gamma)}(\tilde{\Omega})$ can be seen by differentiating the equation (\ref{a11}) and applying mathematical induction. Integrating the expression (\ref{a11.5}) by parts yields the boundary condition. We then change the coordinate back to $\Omega$, and all the conditions remain valid.

For the estimate we observe that from (\ref{a10}) we may solve, for any $l \ge 0$,
\begin{eqnarray}\label{a20}
\partial_{\tilde{x}}^l \partial_{\tilde{y}}^2 u &=& \partial_{\tilde{x}}^l \big[\bar{B}^{-1} (f -\partial_{\tilde{x}}(\bar{K} \partial_{\tilde{x}}u) -\partial_{\tilde{x}}(\bar{A} \partial_{\tilde{y}}u)- \partial_{\tilde{y}}(\bar{A} \partial_{\tilde{x}}u)\\
&& \quad \quad - \bar{C} \partial_{\tilde{x}} u - \bar{D} \partial_{\tilde{y}}u -\partial_{\tilde{y}}(B \partial_{\tilde{y}}u) \big]. \nonumber
\end{eqnarray}

Integrate $Lu = f$ against $\frac{\tilde{y}^2-2}{h^{2 \gamma}} u$ by part as in \lemmaref{l3}, together with Schwartz inequality we have
$$
\| h^{-\gamma - 1} \partial_{\tilde{y}} u \|^2_{L^2(\tilde{\Omega})} +\| h^{-\gamma - 1} u \|^2_{L^2(\tilde{\Omega})} \le C \|h^{-\gamma} f\|^2_{L^2(\tilde{\Omega})}.
$$

Next, we define
$$L_1 (\partial_{\tilde{x}}u) = \partial_{\tilde{x}} L u- (\partial^2_{\tilde{x}}A +\partial_{\tilde{x}}D+ \partial^2_{\tilde{x} \tilde{y}}B)\partial_{\tilde{y}}u.$$
The integration by parts on $ \int_{\tilde{\Omega}} L_1 (\partial_{\tilde{x}} u) \frac{\tilde{y}^2-2}{h^{2 \gamma}} \partial_{\tilde{x}} u$ as in \lemmaref{l3}, together with the estimate on $\partial_{\tilde{y}} u$ and the condition that $|w|_{C^5} \le 1$  yield
$$
\|h^{-\gamma-1} \partial_{\tilde{x} \tilde{y}} u \|^2_{L^2(\tilde{\Omega})} + \|h^{-\gamma-1} \partial_{\tilde{x}} u \|^2_{L^2(\tilde{\Omega})} \le C \|f \|^2_{H^{1, \gamma}(\tilde{\Omega})}.
$$
By a similar method we obtain an estimate for $\partial_{\tilde{x}}^2 u$ and by (\ref{a20}) we obtain an estimate for $\partial^2_{\tilde{y}}u$. The full estimate up to second order is then
\begin{equation} \label{a99}
\|u\|_{H^{2,\gamma+1}(\tilde{\Omega})} \le C \|f\|_{H^{2, \gamma}(\tilde{\Omega})}.
\end{equation}

Then we compute $L(\partial_{\tilde{x}}^s u)$,
\begin{eqnarray}
\quad \partial_{\tilde{x}}(\bar{K} \partial_{\tilde{x}}^{s+1}u) &=& \partial^s_{\tilde{x}}(\partial_{\tilde{x}}( \bar{K}\partial_{\tilde{x}}u))-s \partial_{\tilde{x}}^2 \bar{K} \partial_{\tilde{x}}^su  \label{a100} \\
&&-s \partial_{\tilde{x}}\bar{K} \partial_{\tilde{x}}^{s+1} u - \binom{s}{s-2} \partial_{\tilde{x}}^2 \bar{K} \partial_{\tilde{x}}^{s} u \nonumber \\
&& - \sum_{l=0}^{s-3} \binom{s}{l} \partial_{\tilde{x}}(\partial_{\tilde{x}}^{s-l}\bar{K}\partial_{\tilde{x}}^{l+1}u) - \binom{s}{s-2} \partial_{\tilde{x}}^3 \bar{K} \partial_{\tilde{x}}^{s-1} u, \nonumber \\
\partial_{\tilde{x}}(\bar{A}\partial_{\tilde{y}} \partial_{\tilde{x}}^s u) &=& \partial_{\tilde{x}}^s (\partial_{\tilde{x}}(\bar{A}\partial_{\tilde{y}} u)) -s \partial_{\tilde{x}}^2 \bar{A}\partial_{\tilde{x}}^{s-1}\partial_{\tilde{y}}u -s \partial_{\tilde{x}} \bar{A}\partial_{\tilde{x}}^s \partial_{\tilde{y}} u \label{a101} \\
&& - \sum_{l=0}^{s-2} \binom{s}{l}\partial_{\tilde{x}}(\partial_{\tilde{x}}^{s-l}\bar{A}\partial_{\tilde{x}}^l\partial_{\tilde{y}}u), \nonumber \\
\partial_{\tilde{y}}(\bar{A}\partial_{\tilde{x}}^{s+1}u) &=& \partial_{\tilde{x}}^s( \partial_{\tilde{y}}(\bar{A}\partial_{\tilde{x}}u)) -s\partial_{\tilde{x} \tilde{y}} \bar{A} \partial_{\tilde{x}}^s u -s\partial_{\tilde{x}} \bar{A} \partial_{\tilde{x}}^s \partial_{\tilde{y}} u \label{a102} \\
&& -\sum_{l=0}^{s-2} \binom{s}{l} \partial_{\tilde{y}}(\partial_{\tilde{x}}^{s-l}\bar{A}\partial_{\tilde{x}}^{s+1}u), \nonumber \\
\bar{C} \partial_{\tilde{x}}^{s+1}u  &=& \partial_{\tilde{x}}^s(\bar{C}\partial_{\tilde{x}} u) - s \partial_{\tilde{x}} \bar{C} \partial_{\tilde{x}}^su -\sum_{l=0}^{s-2} \binom{s}{l}\partial_{\tilde{x}}^{s-l}\bar{C}\partial_{\tilde{x}}^{l+1}u, \label{a103} \\
\bar{D} \partial_{\tilde{x}}^s \partial_{\tilde{y}} u &=& \partial_{\tilde{x}}^s (\bar{D} \partial_{\tilde{y}} u) - \sum_{l=0}^{s-1} \binom{s}{l} \partial_{\tilde{x}}^{s-l}\bar{D} \partial_{\tilde{x}}^l \partial_{\tilde{y}} u, \label{a104}\\
\partial_{\tilde{y}}(\bar{B} \partial_{\tilde{y}} \partial_{\tilde{x}}^s u) &=& \partial_{\tilde{x}}^s (\partial_{\tilde{y}}(\bar{B} \partial_{\tilde{y}}u)) -s\partial_{\tilde{x}} \bar{B}\partial_{\tilde{x}}^{s-1}\partial_{\tilde{y}}^2u  \label{a105} \\
&& -s(s-1) \partial_{\tilde{x}}^2 \bar{B} \partial_{\tilde{x}}^{s-2} \partial_{\tilde{y}}^2 u -\sum_{l=0}^{s-1} \binom{s}{l} \partial_{\tilde{x}} ^{s-l}\partial_{\tilde{y}}\bar{B}\partial_{\tilde{x}}^l \partial_{\tilde{y}}u \nonumber\\
&&-\sum_{l=0}^{s-3} \binom{s}{l} \partial_{\tilde{x}}^{s-l}\bar{B}\partial_{\tilde{x}} ^l \partial_{\tilde{y}}^2 u. \nonumber
\end{eqnarray}
For the last equality we need to use (\ref{a20}) to replace the terms with more than one derivative with respect to $\tilde{y}$.

For an arbitrary $s$ we add up the left hand side of (\ref{a100}) through (\ref{a105}) to obtain $L(\partial_{\tilde{x}}^s u)$. On the right hand side, for the terms with $\partial_{\tilde{x}}^s u$, $\partial_{\tilde{x}}^{s+1} u$ and $\partial_{\tilde{x}}^s \partial_{\tilde{y}} u$, we observe that their coefficients are small when $\varepsilon$ is small. Therefore we may move them to the left hand side to form a new operator, then integrate this new operator against $\frac{\tilde{y}^2-2}{h^{\gamma}}\partial_{\tilde{x}}^s u$. The estimate as in \lemmaref{l3} holds by the same type of computation.

The other terms on the right hand side have $\partial_{\tilde{x}}^l f$ with $l \le s$ and the lower derivatives $\partial_{\tilde{x}}^i \partial_{\tilde{y}}^j u$ with $0 \le j \le 1$ and $0 \le i+j < s$. We will make an estimate by integrating them against $\frac{\tilde{y}^2-2}{h^{\gamma}}\partial_{\tilde{x}}^s u$ as well. Below we illustrate our method on one of these terms, the other can be treated in the same way:
$$
\int_{\tilde{\Omega}} \big| \partial_{\tilde{x}}^{s-2}\bar{K} \partial_{\tilde{x}}^{4} u \frac{\tilde{y}^2-2}{h^{2\gamma}} \partial_{\tilde{x}}^s u\big| \le  \frac{4}{ \delta_s} \| (h^{-\gamma+1} \partial_{\tilde{x}}^{s-2} \bar{K} \partial_{\tilde{x}}^{4} u)\|_{L^2(\tilde{\Omega})}^2 + \delta_s \| h^{-\gamma-1} \partial_{\tilde{x}}^s u \|^2_{L^2(\tilde{\Omega})},
$$
where $\delta_s$ is a small constant depending only on $s$ so that $\delta_s \| h^{-\gamma-1} \partial_{\tilde{x}}^s u \|^2_{L^2(\tilde{\Omega})}$ can be absorbed into the left hand side of our estimate.

For
$
\frac{4}{ \delta_s} \| (h^{-\gamma+1} \partial_{\tilde{x}}^{s-2} \bar{K} \partial_{\tilde{x}}^{4} u)\|_{L^2(\tilde{\Omega})}^2,
$
we notice that if $s \le 3$ we use the condition $|w|_{C^4} \le 1 $ and absorb it into the left hand side of our estimate. For the case $s > 3$ we change the coordinate to $(x, y)$ and denote the derivative with respect to this coordinate by $\partial$. Then we extend $u$ by \lemmaref{l3.1} and apply \lemmaref{c3.1} below,
\begin{eqnarray*}
&&\| h^{-\gamma+1} \partial_{\tilde{x}}^{s-2}(\bar{K}) \partial_{\tilde{x}}^4 u \|_{L^2(\tilde{\Omega})} \\
&& \quad \le \| h^{-\gamma+s+1} \partial^{s-2}k \partial^4 u \|_{L^2(\mathbb{R}^2_+)} \le \| \partial^{\gamma-s-1} (\partial^{s-2}k \partial^4 u) \|_{L^2(\mathbb{R}^2_+)} \\
&& \quad \le \| \partial k \|_{L^{\infty}(\mathbb{R}^2_+)} \| u\|_{H^{\gamma}(\mathbb{R}^2_+)}+ \| k\|_{H^{\gamma+1}(\mathbb{R}^2_+)} \| u \|_{L^{\infty}(\mathbb{R}^2_+)},
\end{eqnarray*}
where the last inequality is the Nirenberg inequality (see \lemmaref{l3.3}). With the assumption that $|w|_{C^4} \le 1$, we have $\partial k = O(\varepsilon)$. Apply a similar method on the terms with $f$ and the other coefficients, we arrive at
\begin{eqnarray*}
&&\|h^{-\gamma-1} \partial_{\tilde{y}} \partial_{\tilde{x}}^s u \|_{L^2(\tilde{\Omega})} + \|h^{-\gamma-1} \partial_{\tilde{x}}^s u \|_{L^2(\tilde{\Omega})} \\
&& \quad \le \|f\|_{H^{ \gamma}(\mathbb{R}^2_+)} +\|w\|_{H^{s+6}(\mathbb{R}^2_+)} \|f\|_{H^2(\mathbb{R}^2_+)}+ O(\varepsilon)\|u\|_{H^{\gamma}(\mathbb{R}^2_+)}\\
&& \quad \quad +(1+\|w\|_{H^{s+6}(\mathbb{R}^2_+)} )\|u\|_{H^2(\mathbb{R}^2_+)}.
\end{eqnarray*}
Solving for higher derivatives of $u$ with respect to $\tilde{y}$, we obtain a similar estimate as above.

For the terms in equations (\ref{a99}) through (\ref{a105}), we change the coordinates to $(x,y)$ for every term that has not been changed. Adding all these estimates up to the order $m$ we have
\begin{eqnarray*}
&&\|u\|_{H^{m,\gamma+1}(\Omega)} \\
&& \quad \le C_m (\|f\|_{H^{ \gamma}(\mathbb{R}^2_+)} +\|w\|_{H^{s+6}(\mathbb{R}^2_+)} \|f\|_{H^2(\mathbb{R
}^2_+)}+ O(\varepsilon)\|u\|_{H^{\gamma}(\mathbb{R}^2_+)}\\
&& \quad \quad +(1+\|w\|_{H^{s+6}(\mathbb{R}^2_+)} )\|u\|_{H^2(\mathbb{R}^2_+)}).
\end{eqnarray*}
Lastly, we pick $\gamma = m$ and $\varepsilon$ small depending on $m$ and use the estimate (\ref{a99}) to obtain the desired result.
\end{proof}


\section{The Nash-Moser iteration with elliptic cusps}
In this section we derive the solution of our nonlinear PDE with $K$ positive in a cusp domain. We first provide some basic tools for the Nash-Moser iteration. The following extension theorem is valid in a cusp domain with the norm
$$
\| u \|_{H^{(m,\gamma)}(\Omega)}= \sum_{\substack{t,s \in \mathbb{N} \\0 \le t+s \le m }}  \int_{\Omega} |h^{-\gamma+t+s}  \partial_x^s \partial_y^t u|^2.
$$
\begin{lemma} \label{l3.1}
Suppose $\Omega = \{\ (x,y) \ | \ 0 \le x \le 1,\ \ -h(x) \le y \le h(x) \ \}$ is a cusp domain with $h \in C^{\infty}((0,\infty))$ satisfies
\begin{equation} \label{a3.1}
\big| h^{(l)} \cdot h^{l-1} \big| \le C_l,\ \ \ \ \ \ \text{for } \ l \in \mathbb{N}.
\end{equation}
Then there is an operator
$$
E: H^{(m,\gamma)}(\Omega) \rightarrow H^{(m,\gamma)}(\mathbb{R}^2_+)
$$
with a positive constant $C$ independent of $u$ such that
$$
\|Eu\|_{H^{(m,\gamma)}(\mathbb{R}^2_+)} \le C \|u\|_{H^{(m,\gamma)}(\Omega)} \text{,} \ \ \ \ \text{and} \ \ Eu=u \text{ in } \Omega.
$$
Moreover, we may choose $Eu$ vanishing outside $y= \pm 2h(x)$.
\end{lemma}
\begin{proof}
We may apply a coordinate change
\begin{eqnarray*}
\tilde{x}=\int_{x}^1 \frac{1}{h(x)} \ \ \ \ \text{and} \ \ \tilde{y}=\frac{y}{h(x)}
\end{eqnarray*}
to map the domain to a half infinite cylinder $\tilde{\Omega}= \{ \ (\tilde{x},\tilde{y}) \ |\ 1 \le \tilde{x} \text{, } -1 \le \tilde{y} \le 1 \ \}$ and define the norm
$$
\| u \|^2 _{H^{(m,\gamma)}(\tilde{\Omega})} =\sum_{ \substack{t,s \in \mathbb{N} \\0 \le t+s \le m } } \int_{\tilde{\Omega}} |h^{-\gamma} \partial_{\tilde{x}}^s \partial_{\tilde{y}}^t u|^2.
$$
With the condition (\ref{a3.1}), it is straightforward to see that there exists positive constants $C_1$ and $C_2$ such that
\begin{equation*}
C_1 \|u\|_{H^{(m,\gamma)}(\tilde{\Omega})} \le \|u\|_{H^{(m,\gamma)}(\Omega)} \le C_2 \|u\|_{H^{(m,\gamma)}(\tilde{\Omega})}.\\
\end{equation*}
Then we may define $E$ on $H^{(m,\gamma)}(\tilde{\Omega})$ by extending  $u \in H^{(m,\gamma)}(\tilde{\Omega})$ across the boundary $\{\tilde{y}= \pm 1 \}$ and require that $u$ vanishes outside $\{ |\tilde{y}| \le 2 \}$. Change the coordinate back for the desired result.
\end{proof}
\begin{coro} \label{c3.1}
With the same assumption and the extension $E$ defined in \lemmaref{l3.1}, we have $$\|Eu\|_{H^{(m,\gamma)}(\mathbb{R}^2_+)} \le C_{\gamma} \|Eu\|_{H^{\gamma}(\mathbb{R}^2_+)} $$ for $u \in C^{\infty}(\Omega)$ and vanishes in a neighborhood or the origin.
\end{coro}
\begin{proof}
For $w \in C^{\infty}(\mathbb{R}^2_+)$, vanishes outside $y= \pm 2h$ and also vanishes in a neighborhood of the origin, we may first compute on $\mathbb{R}^2_+ \cap \{y \ge 0 \}$
\begin{eqnarray*}
&&
\int_{0}^1 \int_{0}^{2h(x)}|h^{-\gamma+t+s} \partial_x^s \partial_y^t w |^2 dydx\\
&=&\int_{0}^1 \int_{0}^{2h(x)}h^{2(-\gamma+t+s)}| \partial_x^s \partial_y^t w(x,y) -\partial_x^s \partial_y^t w(x,2h(x))|^2 dydx\\
&=&\int_0^1\int_{0}^{2h(x)}h^{2(-\gamma+t+s)} |  \int_y^{2h(x)} \partial_x^s\partial_y^{t+1}w(x,\theta) d\theta   |^2 dydx\\
&\le& \int_0^1 \int_0^{2h(x)} h^{2(-\gamma+1+t+s)} |\partial_x^s \partial_y^{t+1}w(x, \theta)|^2d\theta dx,
\end{eqnarray*}
where the last inequality is the Holder's inequality. Then our hypothesis on $u$ allows us to find a family $\{ \ w_i \ | \ w_i \in C^{\infty}(\Omega) \ \}$ such that $ w_{i} \rightarrow Eu$ in $H^{\gamma}(\Omega)$ as $i \rightarrow \infty$. The case $y \le 0$ follows similarly.
\end{proof}
We also record the following Gagliardo-Nirenberg's inequality, whose proof can be found in \cite{MR2744149}.
\begin{lemma} \label{l3.3}
Let all $u_i$ below are $C^m(\mathbb{R}^2_+)$\\
(i)If $\alpha_1,....,\alpha_l$ are multi-indices such that $|\alpha_1|+....+|\alpha_l|=m$, then there exists a constant $C_1$ depending on $l$ and $m$ such that
\begin{eqnarray*}
&\| \partial^{\alpha_1}u_1&...\partial^{\alpha_l}u_l\|_{L^2(\mathbb{R}^2_+)} \\
&\le& C \sum_{j=1}^l(|u_1|_{L^{\infty}(\mathbb{R}^2_+)}.... \widehat{|u_j|}_{L^{\infty}(\mathbb{R}^2_+)}...|u_l|_{L^{\infty}(\mathbb{R}^2_+)}) \| u_j \|_{H^m(\mathbb{R}^2_+)},\\
\end{eqnarray*}
where $\widehat{|u_j|}_{L^{\infty}(\mathbb{R}^2_+)}$ denotes the omitted term in the product.\\
(ii) Let $\mathcal{D} \subset \mathbb{R}^l$ be compact and contain the origin, and let $G \in C^{\infty}(\mathcal{D})$. If $u \in H^m(\Omega, \mathcal{D})$, then there exists a constant $C_2$ depending on $m$ such that
$$
\|G \circ u \|_{H^m(\Omega)} \le C_2 |u|_{L^{\infty}(\Omega)}(|G(0)|+\|u\|_{H^m(\Omega)}).
$$
\end{lemma}
The following theorem of smoothing operator is proved in \cite{MR2765727}.
\begin{lemma} \label{l3.4}
Suppose $\overline{H}^l(\mathbb{R}^2_+)$ is the completion of $C^{\infty}$ function vanishing in a neighborhood of y-axis with respect to the usual Sobolev norm, $\| \cdot \|_l$. Given $\mu \ge 1$ there exists a linear smoothing operator $S_{\mu}:L^2(\mathbb{R}^2_+) \rightarrow \overline{H}^{\infty} (\mathbb{R}^2_+)$ such that for all $l,m \in \mathbb{Z}_{ \ge 0}$ and $u \in \overline{H}^l(\mathbb{R}^2_+)$,
\begin{enumerate}
\item $\| S_{\mu} u \|_m \le C_{l,m} \|u \|_l, \ \ \ m \le l,$
\item $\|S_{\mu} u \|_m \le C_{l,m} \mu^{m-l} \|u \|_l, \ \ \ l \le m,$
\item $\| u - S_{\mu} u \|_m \le C_{l,m} \mu^{m-l} \|u \|_l, \ \ \ m \le l.$
\end{enumerate}
We also have the standard smoothing operator $S'_{\mu}: L^2(\mathbb{R}^2_+) \rightarrow H^{\infty}(\mathbb{R}^2_+)$ such that $(1)$ to $(3)$ hold whenever $u \in H^l(\mathbb{R}^2_+)$.
\end{lemma}
We compare the two smoothing operators. The construction of $S_{\mu}$ involves cutting off near the origin. Therefore $S_{\mu}u$ vanishes near the origin and the boundedness requires that $u$ vanishes near the origin. The standard operator $S'_{\mu}$ on the other hand consists only of convolution with smooth functions, so its boundedness does not require vanishing.

To obtain the solution for our nonlinear PDE in the elliptic region, we need to apply the Nash-Moser iteration. The estimate we derived in \theoremref{t2} is the standard form for the iteration, and it is true that the standard procedure works in our case.

\begin{pro} \label{p3.5}
If $m_0 \ge 2N+9$, then we have a sequence,
$$
\{ \ w_n \ | \ \ \ w_n \in \overline{H}^{m_0-8}(\Omega)  \ \ \ \text{and } \ w_n|_{\partial \Omega}=0 \ \},
$$
such that $w_n \rightarrow w$ in $\overline{H}^{m_0-8}(\Omega)$, with $$\|w\|_{m_0-8} \le C \varepsilon^{2N+6}.$$  Furthermore, $\Phi(w_n) \rightarrow 0$ in $C^0(\Omega)$.
\end{pro}
\begin{proof}
This follows from the standard Nash-Moser iteration, whose proof can be found, for example, in \cite{MR816670}. We will present the details of Nash-Moser iteration in the case when $K$ is negative in a cusp domain, which is more complicated. Here we instead make two remarks. First, we require a stronger condition on the solution $w$ that $\|w\|_{m_0-8} \le C \varepsilon^{2N+6}$ for the construction of a smooth solution across the zero set of $K$. This is where the condition $m_0 \ge 2N+9$ is needed. We may then require our approximate solution $z_0$ to satisfy $\| \Phi(z_0)\|_{m_0-2N-8} \le \varepsilon^{2N+6}$. In the details of Nash-Moser iteration we provided later, it will be clear that this $\varepsilon^{2N+6}$ be passed in the iterative process, so that the solution $w$ we obtain is $O(\varepsilon^{2N+6})$ as in the statement.

Secondly, we started the derivation from weighted spaces over $\tilde{\Omega}$ but ended up with the estimate in the regular Sobolev spaces over $\Omega$ and ran Nash-Moser iteration there. The main tool for the transfer is \cororef{c3.1}, whose proof indicates that the coefficients in the estimate depend on the domain where the functional space is defined on. We point out that the domain $\tilde{\Omega}$ depends on the function $w$ at which the equation is linearized. Therefore we always change the coordinates back to $\Omega$ first and then apply \cororef{c3.1}, to avoid  the dependence of the constant on $w$. Similar issues happened when dealing with the case $K$ negative in a cusp domain in the later sections.
\end{proof}

The complement of $\Omega$, in whose interior $K$ is negative, is (at least) a Lipschitz domain and has been dealt with in \cite{MR2765727}. We record the result,
\begin{pro} \label{p3.6}
Suppose $m_0 \ge 3N+19, \phi \in H^{m_0- 9}(\Omega)$ and $\psi \in H^{m_0-10}(\Omega)$. Moreover, suppose
$$
\| \phi \|_{m_0-9, \partial{\Omega}} + \| \psi \|_{m_0-10, \partial{\Omega}} \le C \varepsilon^{2N+6}.
$$
Then there exists a sequence $w_n \rightarrow w$ in $\overline{H}^{m_0-N-13}(\Omega)$. Furthermore $\Phi(w_n) \rightarrow 0$ in $C^0(\Omega)$.
\end{pro}

We require $m_0 \ge 2N+9$ and construct $w_{\kappa}^+ \in \overline{H}^{m_0-8}(\Omega^+_{\kappa})$ by \proref{p3.5} satisfying
\begin{eqnarray*}
&&\Phi(w_{\kappa}^+)=0 \text{ in } \Omega_{\kappa} \text{, } \ \ \ \ \ \ w_{\kappa}^+=0 \text{ on } \partial \Omega_{\kappa}^+.
\end{eqnarray*}
Let $\Omega_{\varrho}^-$ be a region sharing a boundary with $\Omega_{\kappa}^+$. Then by \proref{p3.6}, we may construct $w_{\varrho}^- \in \overline{H}^{m_0-N-13}(\Omega_{\varrho}^-)$ satisfying
\begin{eqnarray*}
&&\Phi(w_{\varrho}^-)=0 \ \text{ in } \ \Omega_{\varrho}^- \text{, } \ \ \ \ \ w_{\varrho}^-|_{\partial \Omega_{\varrho}^-} =0 \ \ \ \  \text{ and }  \ \ \partial_{\nu}w_{\varrho}^-|_{\partial \Omega_{\varrho}^-}=\partial_{\nu}w_{\kappa}^+|_{\partial \Omega_{\varrho}^-},\\
&& \partial^{\tau} w_{\varrho}^-(0,0)=0, \ \ \ \ \ \text{ for } \ \ |\tau| \le m_0-N-15,
\end{eqnarray*}
where we have used the Sobolev embedding theorem and the trace operator. Lastly, by the choice of the approximate solution, the fact that $w_{\kappa}^+$ is in a weighted space and $|h|_{C^1(|x|\le \sigma)}$ decreases to zero as the small positive parameter $\sigma \rightarrow 0$, the common boundary of any pair of $\Omega_{\kappa}^+$ and $\Omega_{\varrho}^-$ is noncharacteristic. So the solutions in each region agree to $m_0-N-15$ along the boundary, and therefore patched to a solution of (\ref{a1}) in a small neighborhood of the origin.

\section{Linear theory in hyperbolic cusps} \label{s.8}
In this section, we derive the existence and regularity for the Cauchy problem of the linearized version of equation (\ref{a1}) in a hyperbolic cusp. In contrast to Section \ref{s2}, we do the case when $K$ is negative in a cusp domain. We need to apply \lemmaref{l1}, with some change of the notations of the linearized equation:

\begin{equation}
L(w)u= \partial_{\bar{x}} (\bar{K} \partial_{\bar{x}})u + \partial^2_{\bar{y}}u+ \bar{C} \partial_{\bar{x}}u + \bar{D} \partial_{\bar{y}} u, \label{a4.1}
\end{equation}
with
\begin{eqnarray*}
\bar{K} &=& K \bar{k}(x,y,w, \nabla w,\nabla ^2 w, \nabla \bar{x}),\\
\bar{C} &=& K \bar{c}(x,y,w, \nabla w,\nabla ^2 w, \nabla ^3 w, \nabla \bar{x}, \nabla ^2 \bar{x})+(a^{22})^{-2} \partial_{x} \Phi (w) \partial_{x} \bar{x} \partial_{\bar{x}} u,\\
\bar{D} &=& \varepsilon^2 \bar{d}(x,y,w, \nabla w,\nabla ^2 w).
\end{eqnarray*}
In addition, hereafter $R^2_+ = \{ (x, y) \ | \ y \ge 0 \}$. We also remark that the intersection of the original domain $\hat{ \Omega }$ with a small neighborhood of the origin may be written in the new coordinates as
$$
\Omega  = \{ \ ({x}, {y}) \ | \ \ \ |{x}| \le \tilde{h}( y ), \ \ 0 \le y \le 1 \ \},
$$
with $\tilde{h}$ is a rescaling of the initial graph $h$ by $\varepsilon^{-2}$. Again for simplicity we will in what follows continue to denote the graph by $h$ instead of $\tilde{h}$. With this notation the condition (\ref{a0}) to (\ref{a0.5}) remain unchanged but (\ref{cusp parameter}) becomes
\begin{equation}\label{cusp parameter h}
h \geq C \varepsilon^{\bar{\alpha}} y^{1+\bar{\alpha}}.
\end{equation}
Moreover, after the coordinate change to $(\bar{x}, \bar{y})$ the domain $\Omega$ becomes,
$$
\bar{ \Omega } = \{ \ (\bar{x}, \bar{y}) \ | \ \ \ |\bar{x}| \le \bar{h}( \bar{y} ), \ \ 0 \le \bar{y} \le 1 \ \}.
$$
Notice that for $\bar{h}$, the constant in the condition (\ref{a0.5}) depends on the function $w$, at which the equation is linearized. This does no harm in the proof of existence but causes trouble for the energy estimate. Our assumption that $|w|_{C^4} \le 1$ removes this dependence for lower derivatives, and our treatment for higher derivatives will be elaborated when they appear.

Equation (\ref{a4.1}) is degenerate hyperbolic. We will instead solve the regularized equation
\begin{eqnarray} \label{a4.3}
&&L_{\theta}u=f \ \text{ in } \bar{\Omega}, \ \ \ \  u|_{\partial \bar{\Omega}_1}=\phi, \ \ \ \ u_y|_{\partial \bar{\Omega}_1}=\psi, \end{eqnarray}
where $L_{\theta}$ differs from $L$ only in that $\bar{K}$ is replaced by $\bar{K}_{\theta}=\bar{K}-\theta \bar{h}'^2$, and $\theta =  | \Phi (w) |_{C^5}$. We will require the Cauchy data $\phi$ and $\psi$ as well as $f$ to vanish to high order at the origin. We denote the bottom part and the upper part of the boundary of $\bar{\Omega}$ by $\partial \bar{\Omega}_1$ and $\partial \bar{\Omega}_2$, respectively. We define $H^{(m,l,\gamma)}(\bar{\Omega})$($H_0^{(m,l,\gamma)}(\bar{\Omega})$) to be the closure of all $C^{\infty}(\bar{\Omega})$ functions (which vanish to all orders at $\partial \bar{\Omega}_1$) with respect to the norm
\begin{equation*}
\| u \|^2_{(m,l,\gamma,\bar{\Omega})}= \int_{\bar{\Omega}} \sum_{\substack{s \le m, t \le l \\ s+t \le \max(m,l)}} | h^{-\gamma+t+s} \partial_{\bar{x}}^s  \partial_{\bar{y}}^t u |^2.
\end{equation*}
We first obtain the following existence theorem for (\ref{a4.3}) with homogeneous Cauchy data and $f$ vanishing to all order at $\partial \bar{\Omega}_1$:
\begin{theorem} \label{t4.2}

Suppose $g \in C^{m_0}$, $w \in C^{\infty}$, $|w|_{C^4(\Omega)} \le 1$ and $f \in H_0^{(m,l,\gamma)}(\bar{\Omega})$. If $ m \le m_0-6 $ and $\varepsilon=\varepsilon(m)$ is sufficiently small, then there exist a weak solution $u_{\theta} \in H^{(m,1,\gamma)}(\bar{\Omega})$ of (\ref{a4.3}) with $\phi, \psi=0$. That is,
$$
(u_{\theta}, L_{\theta}^*v)=(f,v), \ \ \ \text{ for all } \ v \in C^{\infty}(\bar{\Omega})
$$
with $v|_{\partial \bar{\Omega}_2}=v_{\bar{y}}|_{\partial \bar{\Omega}_2}=0$ and $v$ vanishes in a neighborhood of the origin.
\end{theorem}
\begin{proof}

The collection of $v$ as stated in the theorem will be denoted by $\widehat{C}^{\infty}(\bar{\Omega})$ hereafter.
Set
$$
b_s(\bar{x},\bar{y})=h^{-2(\gamma-s)}\bar{K}_{\theta}^{-1}(\bar{x},\bar{y})e^{-\lambda \bar{y}},
$$
and let $\zeta$ be the unique solution of
\begin{eqnarray} \label{a4.5}
&& \sum_{s=0}^m (-1)^{s+1} \partial_{\bar{x}}^s ( b_s \partial_{\bar{x}}^s \zeta_{\bar{y}} ) = v \ \text{ in } \bar{\Omega}, \nonumber \\
&& \zeta |_{\partial \bar{\Omega}_1} = \partial_{\bar{x}}^s \zeta_{\bar{y}} | _{\partial \bar{\Omega}_1} = 0, \ \ \ \ 0 \le s \le m-1,
\end{eqnarray}
with $v \in \widehat{C}^{\infty}(\bar{\Omega})$. The existence and the uniqueness is guaranteed by ODE theory. We then repeatedly differentiate $\partial_{\bar{x}}^s \zeta$ and $\partial_{\bar{x}}^s \zeta_{\bar{y}}$ along $\partial \bar{\Omega}_1$ to obtain,
\begin{equation} \label{a4.6}
\partial_{\bar{x}}^s \partial_{\bar{y}}^t \zeta | _{\partial \bar{\Omega}_1} = 0, \ \ \ \ \ \ s+t \le m \ \ \text{and} \ \ 0 \le t \le 2.
\end{equation}
We then establish an \textit{a priori} estimate
\begin{equation} \label{a4.7}
\left( L_{\theta} \zeta , \sum_{s=0}^m (-1)^{s+1} \partial_{\bar{x}}^s(b_s \partial_{\bar{x}}^s \zeta_{\bar{y}}) \right) \ge C \| \zeta \|^2 _{(m,1,\gamma,\bar{\Omega})}.
\end{equation}
Integrating the left hand side by parts yields,
\begin{eqnarray} \label{a4.8}
\lefteqn{\left( L_{\theta} \zeta , \sum_{s=0}^m (-1)^{s+1} \partial_{\bar{x}}^s(b_s \partial_{\bar{x}}^s \zeta_{\bar{y}}) \right)} \\
&\ge& \sum_{s=0}^m \int_{\bar{\Omega}} \frac{1}{4} (\partial_{\bar{y}} b_s - 2 b_s \bar{D})(\partial_{\bar{x}}^s \partial_{\bar{y}} \zeta)^2 \nonumber \\
&+& \sum_{s=0}^m \int_{\bar{\Omega}} \left[ - \frac{1}{2} \partial_{\bar{y}} (b_s \bar{K}_{\theta}) - 4 \frac{(\partial_{\bar{x}} b_s \bar{K}_{\theta})^2 + (b_s \bar{C})^2}{(\partial_{\bar{y}} b_s - 2 b_s \bar{D})} \right] (\partial_{\bar{x}}^{s+1} \zeta)^2 \nonumber \\
&+& \sum_{s=1}^m \sum_{l=1}^s C_m \int_{\bar{\Omega}} b_s (\partial_{\bar{x}}^s \partial_{\bar{y}} \zeta) \partial_{\bar{x}} (\partial_{\bar{x}}^l \bar{K}_{\theta} \partial_{\bar{x}}^{s+1-l} \zeta ) \nonumber \\
&+& \sum_{s=1}^m \sum_{l=1}^s C_m \int_{\bar{\Omega}} b_s (\partial_{\bar{x}}^s \partial_{\bar{y}} \zeta) (\partial_{\bar{x}}^l \bar{C} \partial_{\bar{x}}^{s+1-l} \zeta) \nonumber \\
&+& \sum_{s=1}^m \sum_{l=1}^s C_m \int_{\bar{\Omega}} b_s (\partial_{\bar{x}}^s \partial_{\bar{y}} \zeta)(\partial_{\bar{x}}^l \bar{D} \partial_{\bar{x}}^{s-1} \partial_{\bar{y}} \zeta) \nonumber \\
&+& \frac{1}{2} \int_{\partial \bar{\Omega}} b_m \bar{K}_{\theta} (\partial_{\bar{x}}^{m+1} \zeta)^2 \nu_2 +2 b_m \partial_{\bar{x}}^{m-1}\partial_{\bar{y}}^2 \zeta \partial_{\bar{x}}^m \partial_{\bar{y}} \zeta \nu_1 - b_m (\partial_{\bar{x}}^m \partial_{\bar{y}} \zeta)^2 \nu_2. \nonumber
\end{eqnarray}
The boundary integral has the correct sign by the usual non-characteristic argument. For the positivity of the first term on the right hand side of the equation (\ref{a4.8}), we compute
\begin{eqnarray}
\lefteqn{\partial_{\bar{y}} b_s - 2b_s \bar{D}} \nonumber \\
&=& -b_s \left( \frac{\lambda}{2}  +2 \bar{D} - \frac{2 \theta \bar{h}' \bar{h}''}{\bar{K}_{\theta}} + 2(\gamma-s) \frac{h'}{h} + \frac{\partial_{\bar{y}} \bar{K}}{\bar{K}_{\theta}} \right) \\
&\ge& -b_s,
\end{eqnarray}
by condition (\ref{a0}) to (\ref{cusp parameter}), and our choice of coordinate.
For the second term on the right hand side of the equation (\ref{a4.8}), consider
\begin{eqnarray*}
\lefteqn{-\frac{1}{4} \partial_{\bar{y}}(b_s \bar{K}_{\theta})(\partial_{\bar{y}} b_s -2 b_s \bar{D}) - 4 (\partial_{\bar{x}} b_s \bar{K}_{\theta})^2 - 4(b_s \bar{C})^2}\\
&\ge& (\bar{K}_{\theta} b_s)^2 \left[ \frac{-\lambda^2 }{\bar{K}_{\theta}}  -C(\gamma) \lambda  \frac{\partial_{\bar{y}} \bar{K}}{\bar{K}^2_{\theta}} -4 \left( \frac{\partial_{\bar{x}} \bar{K}}{\bar{K}_{\theta}} \right) ^2 - 4 \left( \frac{\bar{C}}{\bar{K}_{\theta}} \right) ^2 \right].
\end{eqnarray*}
Because $-\lambda \partial_{\bar{y}} \bar{K}- O((\partial_{\bar{x}} \bar{K})^2) \ge 0$,
the term
$$
-C( \gamma) \lambda  \frac{\partial_{\bar{y}} \bar{K}}{\bar{K}^2_{\theta}}-4 \left( \frac{\partial_{\bar{x}} \bar{K}}{\bar{K}_{\theta}} \right) ^2 \ge 0.
$$
Therefore this part has the correct sign in the inequality, we only need to take care of the other part.

Recall that $\bar{C} = O(K+\partial_x \Phi(w))$, so
$$
\left| \frac{\bar{C}}{\bar{K}_{\theta}} \right| \le \frac{-K+|\partial_x \Phi(w)|}{-\bar{K}+\theta (\bar{h}')^2} \le \frac{-K+h^2|h^{-2} \partial_x \Phi(w)|}{-\bar{K}+\theta (\bar{h}')^2}\le \frac{-K+ \theta h^2}{-\bar{K}+\theta (\bar{h}')^2} \le 1,
$$
where we have used the fact that $\Phi(w)$ vanishes to high order at the origin. Therefore, the part
${-\lambda^2 }/{\bar{K}_{\theta}}-( {\bar{C}}/{\bar{K}_{\theta}} )^2$ is greater than a positive constant. Here we emphasize the fact that only up to the second derivatives of $h$ and $\bar{h}$ are involved so far. Moreover, the computation above requires $\lambda$ to be large but not dependent on $\varepsilon$ or $\theta$. This is important because this computation will be used in the energy estimate.

For the last three interior integals in (\ref{a4.8}), consider the following term first,
\begin{eqnarray*}
&&\Bigg| \int_{\bar{\Omega}} b_s (\partial_{\bar{x}}^s \partial_{\bar{y}} \zeta) \partial_{\bar{x}} \bar{K} \partial_{\bar{x}}^{s+1} \zeta\ \Bigg| \\
&\le& \int_{\bar{\Omega}} h^{-2\gamma+2s} \frac{1}{-\bar{K}_{\theta}} \left[ (\partial_{\bar{x}}^s \partial_{\bar{y}}  \zeta)^2 + (\partial_{\bar{x}} \bar{K} \partial_{\bar{x}}^{s+1} \zeta)^2 \right]\\
&\le& \int_{\bar{\Omega}} h^{-2\gamma+2s} \frac{1}{-\bar{K}_{\theta}} (\partial_{\bar{x}}^s \partial_{\bar{y}}  \zeta)^2 +h^{-2\gamma+2s} \frac{1}{\theta} \Big(\frac{\partial_{\bar{x}} \bar{K}}{\bar{h}'}\Big)^2 (\partial_{\bar{x}}^{s+1} \zeta)^2.
\end{eqnarray*}
Absorb the extra $\bar{h}'^{-2}$ in front of $\partial_{\bar{x}}^{s+1} \zeta$ as follows,
\begin{eqnarray}\label{a4.8.5}
\Big| \frac{\partial_{\bar{x}} \bar{K}}{\bar{h}'} \Big| \le \frac{|\partial_{\bar{x}} \bar{K}(\bar{x},\bar{y}) - \partial_{\bar{x}} \bar{K}(h(\bar{y}),\bar{y})|}{h} \le \frac{\int_{\bar{h}(\bar{y})}^{\bar{x}} |\partial_{\bar{x}}^2 \bar{K}|}{h} \le |\bar{K}|_{C^2}.
\end{eqnarray}
Similarly, there are also such extra $\bar{h}'^{-2}$ in the rest of the integrals. However, in those integrands, the derivatives of $\zeta$ with respect to $\bar{x}$ have order strictly smaller than $s+1$, so the extra $\bar{h}'^{-2}$ can be absorbed to the weighted functions. With $\lambda=\lambda(\varepsilon,\theta)$ sufficiently large, all these integrals would be dominated by the first two integrals of the equation (\ref{a4.8}).  Use the proof in \cororef{c3.1}, we can estimate $\|\xi\|_{H^{0,\gamma}(\bar{\Omega})} $ and arrive at (\ref{a4.7}).

Then we compute
\begin{eqnarray*}
\| v \|_{(-m,0,\gamma,\bar{\Omega})} &:=& \sup_{\eta \in H_0^{(m,0,\gamma)}(\bar{\Omega})} \frac{|(\eta,v)|}{\| \eta \|_{(m,0,\gamma,\bar{\Omega})}} \\
&=& \sup_{\eta \in H_0^{(m,0,\gamma)}(\bar{\Omega})} \frac{|(\eta, \sum_{s=0}^m (-1)^{s+1} \partial_{\bar{x}}^s(b_s \partial_{\bar{x}}^s \partial_{\bar{y}} \zeta))}{\| \eta \|_{(m,0,\gamma,\bar{\Omega})}} \nonumber \\
&\le& \theta^{-1} C \| \zeta \|_{(m,1,\gamma,\bar{\Omega})}, \nonumber
\end{eqnarray*}
and
\begin{eqnarray*}
\| \zeta \|_{(m,1,\gamma,\bar{\Omega})} \|L^*_{\theta} v \|_{(-m,-1,\gamma,\bar{\Omega})} &\ge& (\zeta, L^*_{\theta} v ) = (L_{\theta}\zeta, v ) \\
&=& \bigg( L_{\theta} \zeta, \sum_{s=0}^m (-1)^{s+1} \partial_{\bar{x}}^s (b_s \partial_{\bar{x}}^s \partial_{\bar{y}} \zeta)  \bigg) \ge C \| \zeta \|^2_{(m,1,\gamma,\bar{\Omega})}.
\end{eqnarray*}
Together we have,
\begin{equation} \label{a4.9}
\|v\|_{(-m,0,\gamma,\bar{\Omega})} \le \theta^{-1}C \| L^*_{\theta}v \|_{(-m,-1,\gamma,\bar{\Omega})}, \ \ \ \ \ \ \text{ for all } \ \ v\in \widehat{C}^{\infty}(\bar{\Omega}).
\end{equation}
Define a linear functional $F:X \rightarrow \mathbb{R}$, where $X=L^*_{\theta}\widehat{C}^{\infty}(\bar{\Omega})$, by
$$
F(L^*_{\theta}v)=(f,v).
$$
Then by (\ref{a4.9}) we deduce that $F$ is bounded on $X$ as a subspace of $H^{(-m,-1,\gamma)}_0(\bar{\Omega})$ since
$$
| F(L^*_{\theta}v) | \le \| f \| _{(m,0,\gamma,\bar{\Omega})} \|v \|_{(-m,0,\gamma,\bar{\Omega})} \le \theta^{-1}C \|f\|_{(m,0,\gamma,\bar{\Omega})} \|L^*_{\theta}v \|_{(-m,-1,\gamma,\bar{\Omega})}.
$$
Lastly, we may apply Hahn-Banach theorem to obtain a bounded extension of $F$ defined on $H^{(-m,-1,\gamma)}_0(\bar{\Omega})$. Since $H^{(m,1,\gamma)}_0(\bar{\Omega})$ is a Hilbert space, there exists a unique $u_{\theta} \in H^{(m,1,\gamma)}_0(\bar{\Omega})$ such that
$$
F(\xi)=(u_{\theta},\xi), \ \ \ \ \text{ for all } \ \xi \in H^{(-m,-1,\gamma)}_0(\bar{\Omega}).
$$
By restricting $\xi$ back to $X$, we see that $u_{\theta}$ solves
$$
(u_{\theta}, L_{\theta}^*v)=(f,v), \ \ \ \ \text{ for all } \ v \in \widehat{C}^{\infty}(\bar{\Omega}).
$$
\end{proof}
We define $\| \cdot \|_{(m,\gamma,\bar{\Omega})}=\| \cdot \|_{(m,m,\gamma,\bar{\Omega})}$ and $H_0^{(m,\gamma)}(\bar{\Omega})=H_0^{(m,m,\gamma)}(\bar{\Omega})$, then derive the following corollary.
\begin{coro} \label{c4.3}
Under the hypothesis of \theoremref{t4.2}, if $f\in H^{(m,\gamma)}_0(\bar{\Omega})$ then there exists a unique solution $u_{\theta} \in H^{(m,\gamma)}_0(\bar{\Omega})$ of (\ref{a4.3}) with $\phi, \psi =0$, for each $\theta \ge 0$.
\end{coro}
\begin{proof}
From \theoremref{t4.2}, the weak solution $u_{\theta} \in H_0^{(m,1,\gamma)}(\bar{\Omega})$ is also in $H_0^{(m,1)}(\bar{\Omega})$ as in \cite{MR2765727}. Therefore we may apply their method to deduce  the boundary condition, uniqueness and the fact that $u_{\theta} \in H^m(\bar{\Omega})$. Then with usual boot-strap procedure we conclude that $u_{\theta} \in H_0^{(m,\gamma)}(\bar{\Omega})$.
\end{proof}
We define $\overline{C}^{\infty}(\bar{\Omega})$ be the collection of $C^{\infty}(\bar{\Omega})$ functions vanishing in a neighborhood of the origin and $\overline{H}^{(m,\gamma)}_0(\bar{\Omega})$ to be the completion of $\overline{C}^{\infty}(\bar{\Omega})$ with respect to the norm $\| \cdot \|_{(m,\gamma,\bar{\Omega})}$.

Here we need to pause and introduce a notation, which is motivated by the following estimate with the Nirenberg inequality. Let $\tau$ and $\sigma$ be multi-indices such that $| \tau| =t$ and $|\sigma| =s $, consider
\begin{eqnarray*}
&&\| \partial^{\tau} (\partial^2_{x} \bar{x}) \partial^{\sigma} u \|_{L^2}  \\
&\le& C(\|  (\partial^2_{x} \bar{x})\|_{L^{\infty}} \|u\|_{H^{t+s}} +\|  (\partial^2_{x} \bar{x})\|_{H^{t+s}} \|u\|_{L^{\infty}} )      \\
&\le& C(\|u\|_{H^{t+s}}+\|w\|_{H^{t+s+6}} \|u\|_{L^{\infty}}).
\end{eqnarray*}
We merely used the fact that $\|\partial^2_{x} \bar{x}\|_{H^m} < C (1+ |w|_{H^{m+6}})$ here. This property is actually shared by several functions, such as $a^{ij}$ and $ \partial_xK$ or the coordinate $\partial_x \partial_{\bar{x}} x$. If we replace $ \partial^2_{x} \bar{x} $ by them in the estimate above, the same result can be obtained. Therefore, we intend to use an abstract notation $\Lambda_l$ to replace this family of functions. We agree that $\Lambda_l$ represents a $ C^{m_0-4-l} $ function whose $ H^m $ norm can be estimate by $C_{l,m} (1+\|w\|_{H^{l+m+6}})$. The following are the functions to be replaced by $\Lambda_l$ when appropriate:
\begin{enumerate}
\item An \textit{a priori} estimate for the first order PDE (see \cite{MR2779549}) yields ${\partial} ^{2+l} \bar{x}=\Lambda_l$.
\item By the estimate for the first order PDE, the assumption that $|w|_{C^4} \le 1$ and the Nirenberg inequality (\lemmaref{l3.3}) we have $ \partial^{1+l}(\partial_{\bar{x}}x) = \Lambda_l$.
\item Similarly we have $a^{ij} = \Lambda_0$ and $\bar{K}, \bar{C}, \bar{D}$ defined in (\ref{a4.1}) are $ \Lambda_{-1}, \Lambda_0, \Lambda_{-1}$ respectively.
\item By $|w|_{C^4} \le 1$ and the Nirenberg iequality, we do have $\Lambda_t \cdot \Lambda_s = \Lambda_{t+s}$ and $\partial^t \Lambda_s = \Lambda_{t+s}$ for $t, s \ge 0$. Therefore, an arbitrary product, sum or derivatives of $\Lambda$ function will still be a $\Lambda$ function, except the subscript needs to be modified accordingly.
\end{enumerate}

The next lemma helps us to obtain the existence with more general $f$ and boundary data.
\begin{lemma} \label{l4.4}
Suppose $g\in C^{m_0}$, $\phi \in \overline{C}^{\infty}(\partial \bar{\Omega})$, $\psi \in \overline{C}^{\infty}(\partial \bar{\Omega})$, and $f \in \overline{C}^{\infty}( \bar{\Omega})$ with $m \le m_0-4$. Then there exists a function $\eta_{\theta} \in \overline{H}^{(m,\gamma)}_0(\bar{\Omega})$ such that
\begin{eqnarray} \label{a4.10}
&&\eta_{\theta}|_{\partial \bar{\Omega}_1}=\phi, \ \ \ \ \ \ \partial_{\bar{y}} \eta_{\theta}|_{\partial \bar{\Omega}_1}=\psi, \nonumber \\
&&\partial_{\bar{x}}^t (f-L_{\theta} \eta_{\theta})|_{\partial \bar{\Omega}_1}=0, \ \ \ \ \text{ for } \ 0 \le t \le m-3,
\end{eqnarray}
with
\begin{eqnarray} \label{a4.10.5}
\|\eta_{\theta}\|_{H^{(m,\gamma)}(\bar{\Omega})} &\le& C_{m,\gamma} \Bigg[\sum_{|\tau+\sigma| = 0}^{m-1}\|h^{-\gamma+|\tau+\sigma|}(\frac{\bar{y}}{h})^{1+(m-|\tau+\sigma|)} \Lambda_{|\tau|} \bar{\partial}^{\sigma}f \|_{L^2(\bar{\Omega})} \nonumber \\
&+& \sum_{|\tau+\sigma|=0}^{m}\|h^{-\gamma+|\tau+\sigma|}(\frac{\bar{y}}{h})^{5+(m-|\tau+\sigma|)} \Lambda_{|\tau|} \bar{\partial}^{\sigma}\phi \|_{L^2(\partial \bar{\Omega})} \nonumber \\
&+& \sum_{|\tau+\sigma|=0}^{m-1}\|h^{-\gamma+|\tau+\sigma|}(\frac{\bar{y}}{h})^{5+(m-|\tau+\sigma|)} \Lambda_{|\tau|} \bar{\partial}^{\sigma}\psi \|_{L^2(\partial \bar{\Omega})} \Bigg],
\end{eqnarray}
where $C_{m,\gamma}$ is a constant depending only on $m$ and $\gamma$, $\bar{\partial}$ is derivative with respect to $(\bar{x}, \bar{y})$ and $\tau, \sigma$ are multi-indices.
\end{lemma}
\begin{proof}
For any $u_{\theta}$ satisfying (\ref{a4.10}), since $\partial \bar{\Omega}_1$ is non-characteristic, $u_{\theta}$ can be expressed in terms of $f$, $\phi$ and $\psi$ along the boundary. Then the estimate below follows,

\begin{eqnarray} \label{a4.11}
\|u_{\theta}\|_{H^{(m,\gamma)}(\partial \bar{\Omega})} &\le& C_{_m,\gamma} \Bigg[\sum_{|\tau+\sigma|=0}^{m-2}\|h^{-\gamma+|\tau+\sigma|}(\frac{\bar{y}}{h})^{1+(m-|\tau+\sigma|)}  \Lambda_{|\tau|} \bar{\partial}^{\sigma}f \|_{L^2(\bar{\partial}\bar{\Omega})} \nonumber \\
&+& \sum_{|\tau+\sigma|=0}^{m}\|h^{-\gamma+|\tau+\sigma|}(\frac{\bar{y}}{h})^{5+(m-|\tau+\sigma|)}  \Lambda_{|\tau|} \bar{\partial}^{\sigma}\phi \|_{L^2(\partial \bar{\Omega})} \nonumber \\
&+& \sum_{|\tau+\sigma|=0}^{m-1}\|h^{-\gamma+|\tau+\sigma|}(\frac{\bar{y}}{h})^{5+(m-|\tau+\sigma|)} \Lambda_{|\tau|} \bar{\partial}^{\sigma}\psi \|_{L^2(\partial \bar{\Omega})} \Bigg].
\end{eqnarray}
Then consider coordinate change
$$
\tilde{x}=\frac{\bar{x}}{\bar{h}(\bar{y})}, \ \ \ \ \ \  \tilde{y}=\int_{\bar{y}}^1 \frac{1}{\bar{h}},
$$
which maps $\bar{\Omega}$ into an infinite half cylinder $\tilde{\Omega}= \{ \ (\tilde{x}, \tilde{y}) | \  |\tilde{x}| \le 1, \ \ 1 \le \tilde{y} \ \}.$
Since $\tilde{\Omega}$ is a Lipschitz domain, we have surjectivity from $H^m(\tilde{\Omega})$ to $H^m(\partial \tilde{\Omega})$. Then by quotienting the kernel and applying the closed graph theorem, we obtain a $\bar{\eta}_{\theta}$ such that

$$
\tilde{\partial}^{\tau} \bar{\eta}_{\theta} = \tilde{\partial}^{\tau}(h^{-\gamma} u_{\theta}) \ \text{ on } \ \partial \tilde{\Omega}_1 \ \ \ \ \text{ for } \ |\tau| \le m-3.
$$

We define $\eta_{\theta}=h^{\gamma} \bar{\eta}_{\theta}$,
and then the estimate follows from
$$
\| h^{-\gamma}\eta_{\theta} \| _{H^m(\tilde{\Omega})}=\| \bar{\eta}_{\theta} \| _{H^m(\tilde{\Omega})}
\le C_{m} \sum_{|\tau| \le m} \| \tilde{\partial}^{\tau}(h^{-\gamma} u_{\theta}) \|_{L^2( \partial \tilde{\Omega}_1)},
$$
and changing the coordinate back to $(\bar{x},\bar{y})$.
\end{proof}

Here we remark in \lemmaref{l4.4} we have to estimate the higher derivatives of $\bar{h}$. Observe the following equation along the boundary
\begin{equation} \label{est_shape}
\bar{h} ( \bar{y} ) = \bar{x} (h(y) ,y).
\end{equation}
Repeatedly differentiating this equation shows that the derivatives of $\bar{h}$ can be estimated by the power of $h$ and the derivatives of the coordinate functions $(\bar{x}, \bar{y})$.

Here we revisit the "loss of derivatives" issue we mentioned in the introduction. A well behaved second order differential operator is expected to be an isomorphism $L:H^{m+2} \rightarrow H^{m}$, which means for an arbitrary data $f \in H^{m}$ we can find the unique $H^{m+2}$ solution for it. In the case we only have surjectivity $L:H^{m+2} \rightarrow H^{m+N}$, we say it loses $N$ derivatives and Nash-Moser iteration is often used to defeat the loss. The ${\bar{y}}/{h}$ part in the estimate above is the main source of "loss of weights" and we will apply \lemmaref{c3.1} to absorb them into the loss of derivatives. However, we see from the estimate that the loss of weights depends on $m$. This causes difficulty because Nash-Moser iteration requires an estimate
$$
\|u\|_{H^m} \le C(\|f\|_{H^{am+b}}),
$$
with $a < 2$ (see \cite{MR546504, MR656198}, where the authors provide a counter example to the inverse function theorem when $a=2$).

This forces us to induce a strong condition (\ref{cusp parameter h}) so we can rewrite the estimate (\ref{a4.10.5}) as
\begin{eqnarray}\label{a4.13}
\|\eta_{\theta}\|_{H^{(m,\gamma)}(\bar{\Omega})} &\le& C_{m,\gamma} \Bigg[\sum_{|\tau+\sigma| = 0}^{m-1}\|h^{-\gamma-\alpha m+|\tau+\sigma|} \Lambda_{|\tau|} \bar{\partial}^{\sigma}f \|_{L^2(\bar{\Omega})} \nonumber \\
&+& \sum_{|\tau+\sigma|=0}^{m}\|h^{-\gamma- \alpha m-3 +|\tau+\sigma|} \Lambda_{|\tau|} \bar{\partial}^{\sigma}\phi \|_{L^2(\partial \bar{\Omega})} \nonumber \\
&+& \sum_{|\tau+\sigma|=0}^{m-1}\|h^{-\gamma- \alpha m -3 + |\tau+\sigma|} \Lambda_{|\tau|} \bar{\partial}^{\sigma}\psi \|_{L^2(\partial \bar{\Omega})} \Bigg],
\end{eqnarray}
with $ \alpha= \frac{\bar{\alpha}}{1+\bar{\alpha}}$.
Notice that this estimate improves \lemmaref{l4.4} by allowing a weaker hypothesis $\phi \in \overline{H}^{(m,\gamma+ \alpha m+5)}_0(\partial \bar{\Omega}_1)$, $\psi \in \overline{H}^{(m-1,\gamma+ \alpha m+5)}_0(\partial \bar{\Omega}_1)$ and $f \in \overline{H}^{(m-1,\gamma+ \alpha m)}_0(\bar{\Omega})$.

\begin{theorem} \label{t4.5}
Suppose that $g \in C^{m_0}$, $\phi, \psi \in \overline{C}^{\infty}(\partial \Omega)$, $f \in \overline{C}^{\infty}(\Omega)$ and $|w|_{3N+14} \le \varepsilon^{N+2}$, where $N \le m_0-2$ is the largest integer such that $\partial^{\tau}K(0,0)=0$ for all $|\tau|\le N$. If $m \le \frac{m_0-3N-21}{1+\alpha}$ and $\varepsilon =\varepsilon(m)$ is sufficiently small, then there exists a unique solution $u_{\theta} \in \overline{H}^{(m,\gamma)}(\Omega)$ of (\ref{a4.3}) for each $\theta \ge 0$. Furtheremore, there exists a constant $C_m$ independent on of $\varepsilon$ and $\theta$ such that
\begin{eqnarray*}
\|u_{\theta}\|_{H^{(m,m)}(\Omega)}&\le& C_m \varepsilon^{-N-1} \Big( \| \phi\|_{H^{(1+\alpha)m+3N+14}(\partial \Omega)} \\
&&+\| \psi\|_{H^{(1+\alpha)m+3N+13}(\partial \Omega)} + \| f\|_{H^{(1+\alpha)m+3N+10}(\mathbb{R}^2_+)}         \Big) \\
&&+C_m \varepsilon^{-N-1} \|w\|_{H^{(1+\alpha)m+3N+19}} \Big(\|\phi\|_{H^{2N+13}(\partial \Omega)} \\
&&+\| \psi\|_{H^{2N+12}(\partial \Omega)} + \| f\|_{H^{2N+9}(\mathbb{R}^2_+)}     \Big).
\end{eqnarray*}
\end{theorem}
\begin{proof}
For all $\phi$, $\psi$ and $f$ as in the statement, we have  $\phi \in \overline{H}^{(m_0-4,\gamma+m_0+3)}_0(\partial \bar{\Omega}_1)$, $\psi \in \overline{H}^{(m_0-5,\gamma+m_0+3)}_0(\partial \bar{\Omega}_1)$ and $f \in \overline{H}^{(m_0-5,\gamma+m_0-2)}_0(\bar{\Omega})$ by a coordinate change. Let $\eta_{\theta} \in \overline{H}^{(m_0-4,\gamma+2)}(\bar{\Omega})$ be as in \lemmaref{l4.4} and $\bar{u}_{\theta} \in H_0^{(m_0-6,\gamma)}(\bar{\Omega})$ be as in \cororef{c4.3} with $\hat{f}=f -L_{\theta} \eta_{\theta} \in H^{(m_0-6,\gamma)}_0(\bar{\Omega})$, then clearly, upon coordinate change, $u_{\theta} = \bar{u}_{\theta}+\eta_{\theta}$ establishes the existence. Uniqueness follows the estimate, which is obtained as follows.

We first set $v_{\theta}=e^{-\frac{1}{2}\bar{y}^2} u_{\theta}$ and observe that
\begin{eqnarray*}
\overline{L}_{\theta} v_{\theta}&:=& \partial_{\bar{x}}(\bar{K}_{\theta} \partial_{\bar{x}} v_{\theta}) + \partial^2_{\bar{y}} v_{\theta} +\bar{C} \partial_{\bar{x}} v_{\theta} +(2 \bar{y}+\bar{D}) \partial_{\bar{y}} v_{\theta}\\
&&+(1+\bar{y}^2+\bar{y}\bar{D}) v_{\theta} =e^{-\frac{1}{2}\bar{y}^2}f := \bar{f}.
\end{eqnarray*}
Then define the operator
\begin{eqnarray*}
\overline{L}_{\theta}^{(m)} v &:=& \partial_{\bar{x}}(\bar{K}_{\theta}\partial_{\bar{x}} v) + \partial^2_{\bar{y}} v +(\bar{C}+m \partial_{\bar{x}}\bar{K}_{\theta})\partial_{\bar{x}}v +(2 \bar{y}+\bar{D})\partial_{\bar{y}}v\\
&&+ \Big( 1+\bar{y}^2 + \bar{y} \bar{D} + m \partial_{\bar{x}} \bar{C} + \frac{m(m+1)}{2} \partial_{\bar{x}}^2 \bar{K}_{_\theta} \Big) v.
\end{eqnarray*}
Thus,
\begin{eqnarray*}
\overline{L}_{\theta}^{(m)} \partial_{\bar{x}}^m v_{\theta}&=& \partial^m_{\bar{x}} \bar{f} - \sum_{s=3}^m \binom{m}{s} \partial_{\bar{x}}^s \bar{K}_{\theta} \partial_{\bar{x}}^{m-s+2} v_{\theta} \\
&&-\sum_{s=2}^m \binom{m}{s} \partial_{\bar{x}}^s (\bar{C}+ \partial_{\bar{x}} \bar{K}_{\theta}) \partial_{\bar{x}}^{m-s+1} v_{\theta}\\
&&- \sum_{s=1}^m \binom{m}{s}[\partial_{\bar{x}}^s \bar{D} \partial_{\bar{x}}^{m-s}(\partial_{\bar{y}} v_{\theta})+ \bar{y} \partial_{\bar{x}}^s \bar{D} \partial_{\bar{x}}^{m-s} v_{\theta}]\\
&:=& \partial_{\bar{x}}^m \bar{f} + \bar{f}^{(m)}(v_{\theta}).
\end{eqnarray*}
We first assume $m \le 2N+6$. Let $\eta_{\theta}$ be as in \lemmaref{l4.4} with $\eta_{\theta} |_{\partial \bar{\Omega}_1}=v_{\theta}|_{\partial \bar{\Omega}_1}$, $\partial_{\bar{y}} \eta_{\theta}|_{\partial \bar{\Omega}_1} = \partial_{\bar{y}} v_{\theta}|_{\partial \bar{\Omega}_1}$ and
\begin{equation*}
\partial_{\bar{x}}^t (\bar{f}-\overline{L}_{\theta} \eta_{\theta})|_{\partial \bar{\Omega}_1}=0, \ \ \ \ \ \ 0 \le t \le m+N+1.
\end{equation*}
So that we have $\bar{\partial}^{\tau} \eta_{\theta}|_{\partial \bar{\Omega}_1}=\bar{\partial}^{\tau} v_{\theta}|_{\partial \bar{\Omega}_1}$ for all $|\tau| \le m+N+3$. Cosider the function $\bar{v}_{\theta}:= v_{\theta}-\eta_{\theta}$ , which satisfies
$$
\overline{L}_{\theta}^{(m)} \partial_{\bar{x}}^m \bar{v}_{\theta} = \partial_{\bar{x}}^m(\bar{f}-\overline{L}_{\theta}\eta_{\theta})+\bar{f}^{(m)}(\bar{v}_{\theta}).
$$
Let $b_m=h^{-2\gamma+ 2m}\bar{K}_{\theta}^{-1}e^{-\lambda \bar{y}}$ and integrate the following by parts:
\begin{eqnarray}\label{a4.15}
&&(-b_m \partial_{\bar{}y}\partial_{\bar{x}}^m \bar{v}_{\theta}, \overline{L}_{\theta}^{(m)} \partial_{\bar{x}}^m \bar{v}_{\theta})= \\
&& \quad \quad \quad \quad \quad \int_{\bar{\Omega}} \Big[ -\frac{1}{2}\partial_{\bar{y}}(b_m \bar{K}_{\theta}) (\partial_{\bar{x}}^{m+1} \bar{v}_{\theta})^2   \nonumber \\
&& \quad \quad \quad \quad \quad \quad \quad +(\partial_{\bar{x}} b_m \bar{K}_{\theta}- m b_m \partial_{\bar{x}}\bar{K}_{\theta}-b_m \bar{C}) \partial_{\bar{x}}^{m+1} \bar{v}_{\theta} \partial_{\bar{y}}\partial_{\bar{x}}^m \bar{v}_{\theta}    \Big] \nonumber \\
&& \quad \quad \quad \quad +\int_{\bar{\Omega}} \Big(  \frac{1}{2} \partial_{\bar{y}} b_m - b_m (2 \bar{y}+\bar{D})  (\partial_{\bar{y}} \partial_{\bar{x}}^m \bar{v}_{\theta}) \Big)^2 \nonumber \\
&& \quad \quad \quad \quad + \int_{\bar{\Omega}} \frac{1}{2} \partial_{\bar{y}}\Big[ b_m \Big( 1+ \bar{y}^2 +\bar{y} \bar{D}     m \partial_{\bar{x}} \bar{C} +\frac{m(m+1)}{2} \partial_{\bar{x}}^2   \bar{K}_{\theta}   \Big) \Big](\partial_{\bar{x}}^m \bar{v}_{\theta})^2 \nonumber \\
&& \quad \quad \quad \quad +\int_{\partial \bar{\Omega}} \Big[ \frac{1}{2}b_m \bar{K}_{\theta} (\partial_{\bar{x}}^{m+1}\bar{v}_{\theta})^2 \nu_2 - b_m \bar{K}_{\theta} \partial_{\bar{x}}^{m+1} \bar{v}_{\theta} \partial_{\bar{y}}\partial_{\bar{x}}^m \bar{v}_{\theta} \nu_1 \Big] \nonumber \\
&& \quad \quad \quad \quad - \int_{\partial \bar{\Omega}} \frac{1}{2} b_m \Big[ \frac{1}{2} b_m (\partial_{\bar{y}} \partial_{\bar{x}}^m \bar{v}_{\theta})^2 \nu_2   \nonumber \\
&& \quad \quad \quad \quad \quad \quad \quad + \Big( 1+\bar{y}^2 +\bar{y}\bar{D}+m \partial_{\bar{x}} \bar{C} +\frac{m(m+1)}{2} \partial_{\bar{x}}^2 \bar{K}_{\theta}     \Big)(\partial_{\bar{x}}^m \bar{v}_{\theta})^2 \nu_2 \Big]. \nonumber
\end{eqnarray}
The boundary integral has the correct sign along $\partial \bar{\Omega}_2$ and vanishes along $\partial \bar{\Omega}_1$. The same computation as in (\ref{a4.7}) yields
\begin{eqnarray*}
&&\lambda (\|h^{-\gamma+m} \partial_{\bar{x}}^{m+1} \bar{v}_{\theta} \|+\|\sqrt{|b_m|} \partial_{\bar{y}} \partial_{\bar{x}}^m \bar{v}_{\theta}\|+  \| \sqrt{|b_m|} \partial_{\bar{x}}^m\bar{v}_{\theta}\|) \\
&& \quad \quad \le C(\|\sqrt{|b_m|} \partial_{\bar{x}}^m (\bar{f}-\overline{L}_{\theta}\eta_{\theta}) \| + \|  \sqrt{|b_m|}\bar{f}^{(m)}(\bar{v}_{\theta}) \|).
\end{eqnarray*}
With the condition $m \le 2N+6$ and $|w|_{C^{2N+10}} \le 1$ we have
$$
\| \sqrt{|b_m|} \bar{f}^{(m)}(\bar{v}_{\theta}) \| \le C_m \sum_{s=0}^{m-1} (\| \sqrt{|b_s|} \partial_{\bar{x}}^s \bar{v}_{\theta} \| +\|  \sqrt{|b_s|} \partial_{\bar{y}} \partial_{\bar{x}}^s \bar{v}_{\theta}    \|).
$$
Then we exploit the fact that $\bar{K}$ vanishes to order $N$ at the origin, there is a constant $C$ such that $|\bar{K}| \ge C ^{-1} \varepsilon^{2N+2} (h(\bar{y}) - |\bar{x}|  )^{N+1}$ in $\bar{\Omega}$. A similar computation as in (\ref{a4.8.5}) shows that
\begin{eqnarray*}
\| \sqrt{|b_m|}\partial_{\bar{x}}^m (\bar{f}-\overline{L}_{\theta}\eta_{\theta}) \|^2\le C \varepsilon^{-2N-2} \int_{\bar{\Omega}}[h^{-\gamma+m} \partial_{\bar{x}}^{m+N+1}(\bar{f}-\overline{L}_{\theta}\eta_{\theta})]^2.
\end{eqnarray*}
Hereafter we use $\{\tau_i \ | \ i \in \mathbb{N} \}$ to denote multi-indices. Applying (\ref{a4.13}) and summing from 0 to m produces
\begin{eqnarray}
&&\sum_{s=0}^{m+1}\|h^{-\gamma+s} \partial_{\bar{x}}^{s+1} \bar{v}_{\theta} \|^2+ \sum_{s=0}^m(\|\sqrt{|b_s|} \partial_{\bar{y}} \partial_{\bar{x}}^s \bar{v}_{\theta}\|^2+  \| \sqrt{|b_s|} \partial_{\bar{x}}\bar{v}_{\theta}\|^2) \\
&& \quad \le C_m \varepsilon^{-N-1} \Big[\int_{\partial \bar{\Omega}_1} \sum_{|\tau_1+\tau_2|=0}^{m+N+3}  (h^{-\gamma-\alpha m -(N+3)-5+|\tau_1+\tau_2|}\Lambda_{|\tau_1|} \bar{\partial}^{\tau_2} \phi )^2 \nonumber\\
&& \quad \quad \quad \quad\int_{\partial \bar{\Omega}_1} \sum_{|\tau_1+\tau_2|=0}^{m+N+2} (h^{-\gamma-\alpha m-(N+3)-5+|\tau_1+\tau_2|}\Lambda_{|\tau_1|} \bar{\partial}^{\tau_2} \psi )^2 \nonumber \\
&& \quad \quad \quad \quad \int_{\bar{\Omega}} \sum_{|\tau_1+\tau_2|=0}^{m+N+2} \int_{\bar{\Omega}} (h^{-\gamma-\alpha m -(N+3)+|\tau_1+\tau_2|}\Lambda_{|\tau_1|} \bar{\partial}^{\tau_2} f )^2 \Big], \nonumber
\end{eqnarray}
when $\lambda= \lambda(m)$ is sufficiently large. From here it is straightforward to estimate all the $\bar{x}$-derivatives of $u_{\theta}$, then the usual boot-strap procedure yields the desired estimate in this special case.

We then assume $m \ge 2N +7$. We define recursively,
\begin{eqnarray*}
\bar{f}^{(m)}_1(v_{\theta})&=& - \sum_{s=2N +7}^m \binom{m}{s} \partial_{\bar{x}}^s \bar{K}_{\theta} \partial_{\bar{x}}^{m-s+2}v_{\theta}\\
&&-\ \sum_{s=2N +6}^m \binom{m}{s} \partial_{\bar{x}}^s (\bar{C}+ \partial_{\bar{x}}\bar{K}_{\theta}) \partial_{\bar{x}}^{m-s+1}v_{\theta}\\
&&- \sum_{s=2N +6}^m \binom{m}{s} \partial_{\bar{x}}^s \bar{D} \partial_{\bar{y}}\partial_{\bar{x}}^{m-s}v_{\theta} -\bar{y}\sum_{s=2N +5}^m \binom{m}{s}\partial_{\bar{x}}^s \bar{D}\partial_{\bar{x}}^{m-s}v_{\theta} \\
&&-\sum_{s=3}^{2N +6}\binom{m}{s} \partial_{\bar{x}}^s \bar{K}_{\theta} \eta_{\theta}^{(m-s+2)}-\sum_{s=2}^{2N +5} \binom{m}{s} \partial_{\bar{x}}^s(\bar{C}+\partial_{\bar{x}}\bar{K}_{\theta}) \eta_{\theta}^{(m-s+1)}\\
&&-\sum_{s=1}^{2N +5} \binom{m}{s} \partial_{\bar{x}}^s \bar{D} \partial_{\bar{y}} \eta_{\theta}^{(m-s)}-\bar{y} \sum_{s=1}^{2N +4} \binom{m}{s} \partial_{\bar{x}}^s \bar{D} \eta_{\theta}^{(m-s)},
\end{eqnarray*}
and
\begin{eqnarray*}
\bar{f}^{(m)}_2(v_{\theta})&=&-\sum_{s=3}^{2N+6} \binom{m}{s} \partial_{\bar{x}}^s \bar{K}_{\theta} v^{(m-s+2)}_{\theta}\\
&&-\sum_{s=2}^{2N+5} \binom{m}{s} \partial_{\bar{x}}^s (\bar{C}+\partial_{\bar{x}} \bar{K}_{\theta}) v^{(m-s+1)}_{\theta}\\
&&-\sum_{s=1}^{2N+5}\binom{m}{s} \partial_{\bar{x}}^s \bar{D} \partial_{\bar{y}} v_{\theta}^{(m-s)}-\bar{y}\sum_{s=1}^{2N+4} \binom{m}{s} \partial_{\bar{x}}^s\bar{D} v_{\theta}^{(m-s)},
\end{eqnarray*}
where $v^{(s)}:= \partial_{\bar{x}}^s v_{\theta}-\eta_{\theta}^{(s)}$ with $\eta_{\theta}^{(s)}=\partial_{\bar{x}}^s \eta_{\theta}$ for $0 \le s \le 2N+6$ and for $s \ge 2N+7  $, $\eta_{\theta}^{(s)}$ is given by \lemmaref{l4.4} such that $\eta_{\theta}^{(s)}|_{\partial \bar{\Omega}_1}=\partial_{\bar{x}}^sv_{\theta}|_{\partial \bar{\Omega}_1}$, $\partial_{\bar{y}} \eta_{\theta}^{(s)}|_{\partial \bar{\Omega}_1}=\partial_{\bar{y}}\partial_{\bar{x}}^sv_{\theta}|_{\partial \bar{\Omega}_1}$ with
\begin{equation}
\partial_{\bar{x}}^t (\partial_{\bar{x}}^s\bar{f}+\bar{f}^{(s)}_1(v_{\theta})-\overline{L}^{(s)}_{\theta}\eta_{\theta}^{(s)})|_{\partial \bar{\Omega}_1}=0, \ \ \ \ \text{ for } \ 0 \le t \le N+1.
\end{equation}
Note that $\bar{f}^{(m)}=\bar{f}^{(m)}_1+\bar{f}^{(m)}_2$ under this definition. Because $\phi$, $\psi$ and $f$ vanish in a neighborhood of the origin, $u_{\theta}$ may be chosen from a weighted Sobolev space with arbitrarily large weight. We only need to make sure sufficient differentiability to construct $\eta^{(s)}_{\theta}$. Since
$$\bar{f}_1^{(s)} \in \overline{H}^{\text{min}(m_0-s-6,N+3)}(\bar{\Omega}),$$ $$\partial_{\bar{x}}^s v_{\theta} |_{\partial \bar{\Omega}_1} \in \overline{H}^{m_0-s-7}(\partial \bar{\Omega}_1),$$ and $$\partial_{\bar{y}} \partial_{\bar{x}}^s v_{\theta} |_{\partial \bar{\Omega}_1} \in \overline{H}^{m_0-s-8}(\partial \bar{\Omega}_1),$$
we must have $N+3 \le m_0-s-8$.

For $2N+7 \le s \le m$, $v^{(s)}_{\theta}$ satisfies
$$
\overline{L}^{(s)}_{\theta} v^{(s)}_{\theta}=(\partial_{\bar{x}}^s \bar{f}+\bar{f}^{(s)}_1(v_{\theta})-\overline{L}^{(s)}_{\theta} \eta_{\theta}^{(s)}) +\bar{f}^{(s)}_2(v_{\theta}).
$$
Then (\ref{a4.15}) results in
\begin{eqnarray} \label{a4.18}
&&\lambda(\|h^{-\gamma+m}\partial_{\bar{x}} v^{(s)}_{\theta} \|+ \| \sqrt{|b_s|} \partial_{\bar{y}} v_{\theta}^{(s)}\| + \| \sqrt{|b_s|} v^{(s)}_{\theta}\| )\\
&& \quad \le C(\| \sqrt{|b_s|}(\partial_{\bar{x}}^s \bar{f}+\bar{f}^{(s)}_1(v_{\theta})-\overline{L}^{(s)}_{\theta}\eta^{(s)}_{\theta})\| + \| \sqrt{|b_s|}\bar{f}^{(s)}_2 (v_{\theta})\|        ). \nonumber
\end{eqnarray}
First we compute with $|w|_{C^{2N+11}} \le 1$ that
\begin{eqnarray}
\| \sqrt{|b_s|}\bar{f}^{(s)}_2 (v_{\theta}) \|^2 &\le& C_s \int_{\bar{\Omega}}e^{-\lambda \bar{y}} |\bar{K}_{\theta}|^{-1} h^{-2\gamma+2s} \sum_{l=0}^{s-1} [(v^{(l)}_{\theta})^2 + (\partial_{\bar{y}}v^{(l)}_{\theta})^2]\\
&\le& C_s \sum_{l=3N+11}^{s-1} \int_{\bar{\Omega}}e^{-\lambda \bar{y}} |\bar{K}_{\theta}|^{-1} h^{-2\gamma+2l}[(v^{(l)}_{\theta})^2 + (\partial_{\bar{y}}v^{(l)}_{\theta})^2] \nonumber \\
&& + C_s \int_{\bar{\Omega}} \sum_{l=0}^{3N+10}e^{-\lambda \bar{y}} |\bar{K}_{\theta}|^{-1} h^{-2\gamma+2l}[(\partial_{\bar{x}}^l \bar{v}_{\theta})^2+(\partial_{\bar{y}} \partial_{\bar{x}}^l \bar{v}_{\theta})^2]. \nonumber
\end{eqnarray}
Again as in the previous case we compute:
\begin{eqnarray} \label{a4.20}
&&\| \sqrt{|b_s|}(\partial_{\bar{x}}^s \bar{f}+\bar{f}^{(s)}_1(v_{\theta})-\overline{L}^{(s)}_{\theta} \eta_{\theta}^{(s)})\|^2\\
&& \quad\le  C \varepsilon^{-2N-2} \int_{\bar{\Omega}}[h^{-\gamma+s} \partial_{\bar{x}}^{N+1}(\partial_{\bar{x}}^s \bar{f} +\bar{f}^{(s)}_1(v_{\theta})-\overline{L}^{(s)}_{\theta} \eta^{(s)}_{\theta}) ]^2.\nonumber
\end{eqnarray}
A direct computation shows it suffices to estimate the following term in (\ref{a4.20}):
\begin{eqnarray*}
\| h^{-\gamma+s} \partial_{\bar{x}}^{N+1}\overline{L}^{(s)}_{\theta} \eta^{(s)}_{\theta} \|^2 &\le& \sum_{|\tau|=0}^{N+3} \| h^{-\gamma+s-(N+3)+|\tau|} \bar{\partial}^{\tau} \eta_{\theta}^{(s)}\|^2 \\
&\le&\sum_{|\tau|=0}^{N+3} \| h^{-\gamma+s-(N+3)+|\tau|} \bar{\partial}^{\tau} \eta_{\theta}^{(s)}\|_{L^2(\partial \bar{\Omega}_1)}^2.
\end{eqnarray*}
By the computation for \lemmaref{l4.4}:
\begin{eqnarray*}
&&\sum_{|\tau|=0}^{N+3} \| h^{-\gamma+s-(N+3)+|\tau|} \bar{\partial}^{\tau} \eta_{\theta}^{(s)}\|_{L^2(\partial \bar{\Omega}_1)}^2  \\
&&\quad \le \sum_{|\tau_1+\tau_2|=0}^{s+N+3}\| h^{-\gamma  -\alpha s - 3(N+3)-5 + |\tau_1+\tau_2|} \Lambda_{|\tau_1|} \bar{\partial}^{\tau_2} \phi  \|^2_{L^2(\partial \bar{\Omega}_1)}\\
&&\quad \quad +\sum_{|\tau_1+\tau_2|=0}^{s+N+2}\| h^{-\gamma -\alpha s - 3(N+3)-5 + |\tau_1+\tau_2|} \Lambda_{|\tau_1|} \bar{\partial}^{\tau_2} \psi  \|^2_{L^2(\partial \bar{\Omega}_1)}\\
&&\quad \quad +\sum_{|\tau_1+\tau_2|=0}^{s+N+2}\| h^{-\gamma -\alpha s - 3(N+3) + |\tau_1+\tau_2|} \Lambda_{|\tau_1|} \bar{\partial}^{\tau_2} \bar{f}  \|^2\\
&&\quad \quad +\sum_{|\tau_1+\tau_2|=0}^{N+3}\| h^{-\gamma+s- (1+\alpha)(N+3) + |\tau_1+\tau_2|} \Lambda_{|\tau_1|} \bar{\partial}^{\tau_2} \bar{f}^{(s)}_1 (v_{\theta}) \|_{L^2(\partial \bar{\Omega})}^2.
\end{eqnarray*}
Here we pay special attention to the last term above, because it produces terms to be absorbed into the left hand side of (\ref{a4.18}),
\begin{eqnarray}
&& \sum_{|\tau_1+\tau_2|=0}^{N+3}\| h^{-\gamma+s- (1+\alpha)(N+3) + |\tau_1+\tau_2|} \Lambda_{|\tau_1|} \bar{\partial}^{\tau_2} \bar{f}^{(s)}_1 (v_{\theta}) \|_{L^2(\partial \bar{\Omega})}  \nonumber \\
&&\quad \le  \sum_{|\tau_1 + \tau_2|=0}^{N+3} \sum_{l=2N+7}^s\| h^{-\gamma+s- (1+\alpha)(N+3) + |\tau_1+\tau_2|} \Lambda_{\tau_1} \bar{\partial}^{\tau_2} (\partial_{\bar{x}}^l \bar{K}_{\theta} \partial_{\bar{x}}^{s-l+2} v_{\theta})\| \nonumber \\
&&\quad  + \sum_{|\tau_1+\tau_2|=0}^{N+3} \sum_{l=3}^{2N+6}\| h^{-\gamma+s- (1+\alpha)(N+3) +|\tau_1 +\tau_2|} \Lambda_{\tau_1} \bar{\partial}^{\tau_2} (\partial_{\bar{x}}^l \bar{K}_{\theta} \eta^{(s-l+2)}_{\theta})\|_{L^2(\partial \bar{\Omega})} \nonumber \\
&&\quad  + \text{ \ other integrals from $\bar{f}_1^{(s)}(v_{\theta})$}. \label{a4.20.5}
\end{eqnarray}
We only cope with the integrals with $\bar{K}_{\theta}$, the rest can be treated similarly. First consider the integral with $l \ge 2N+7$. We change the coordinate to $(x,y)$ and then apply \lemmaref{l3.1} to extend $v_\theta$ to $\mathbb{R}^2_+$. Since the extended $v_{\theta}$ vanishes outside $\{ -2h(y) \le x \le 2 h(y) \}$, we may apply \cororef{c3.1}
\begin{eqnarray*}
&&\varepsilon^{-N-1} \sum_{|\tau_1 +\tau_2|=0}^{N+3}\| h^{-\gamma+s- 2(N+3)+1 + |\tau_1+\tau_2|} \Lambda_{\tau_1} \bar{\partial}^{\tau_2} (\partial_{\bar{x}}^l \bar{K}_{\theta} \partial_{\bar{x}}^{s-l+2} v_{\theta})\| \\
&& \quad \le\varepsilon^{-N-1} \sum_{|\tau_1 + \tau_2 +\tau_3|=0}^{s-2N-9} \|\partial^{\gamma-s+2(N+3)} \partial^{\tau_1}(\partial_{\bar{x}}^{2N+7}\bar{K}) \Lambda_{\tau_2} \partial^{\tau_3}v_{\theta}      \|_{L^2(\mathbb{R}^2_+)}\\
&& \quad \le \varepsilon^{-N-1} \Big(|\partial_{\bar{x}}^{2N+7} \bar{K} |_{L^{\infty}(\mathbb{R}^2_+)} \| v_{\theta}  \|_{H^{\gamma+1}(\mathbb{R}^2_+)}+\| \bar{K}  \|_{H^{\gamma+2N+8}(\mathbb{R}^2_+)} \|v_{\theta} \|_{L^{\infty}(\mathbb{R}^2_+)} \Big).
\end{eqnarray*}

We want to absorb the term
$$\varepsilon^{-N-1}|\partial_{\bar{x}}^{2N+7} \bar{K} |_{L^{\infty}(\mathbb{R}^2_+)} \| v_{\theta}  \|_{H^{\gamma+1}(\mathbb{R}^2_+)},$$
as well as other similar terms generated from $f_1^{(s)}$, to the left hand side of (\ref{a4.18}). To do that, we need
\begin{eqnarray} \label{a4.21}
&&\varepsilon^{-N-1}|\partial_{\bar{x}}^{2N+7} \bar{K} |_{L^{\infty}(\mathbb{R}^2_+)}, \ \ \ \ \ \ \ \ \ \ \ \ \ \ \varepsilon^{-N-1}|\partial_{\bar{x}}^{2N+6} \bar{C} |_{L^{\infty}(\mathbb{R}^2_+)}, \nonumber \\
&&\varepsilon^{-N-1}|\partial_{\bar{x}}^{2N+6} \bar{D} |_{L^{\infty}(\mathbb{R}^2_+)}, \ \ \ \ \ \ \text{ and } \ \ \varepsilon^{-N-1}| \bar{y}\partial_{\bar{x}}^{2N+5} \bar{D} |_{L^{\infty}(\mathbb{R}^2_+)}
\end{eqnarray}
to be small. Observe that all these $\bar{K}, \bar{C} \text{ and } \bar{D}$ consist of two families of functions. One of them are originally defined on $(\hat{x}, \hat{y})$. The functions $g$, $\Gamma_{ij}^t$, $K$ and $z_0$ are of this family. Every time they are differentiated with respect to $\bar{x}$, $\varepsilon^2$ will be produced. The other family consists of $w$, $\bar{x}$, $\bar{y}$ and ${a^{ij}}$. We see that $w$ itself is small by the assumption $|w|_{C^{3N+14}} \le \varepsilon^{N+2}$. The function $a^{ij}= \nabla_{ij}z_0 + \varepsilon^5 \nabla_{ij}w$ also produces $\varepsilon^2$ (until it reaches $\varepsilon^{N+2}$) each time it is differentiated with respect to $\bar{x}$. For $\bar{x}$, consider the second derivative $\frac{\partial^2 \bar{x}}{\partial x^2}$ satisfies
$$
\frac{\partial}{\partial y} \Big(\frac{\partial^2 \bar{x}}{\partial x^2}\Big)+\frac{a^{12}}{a^{22}}\frac{\partial}{\partial x} \Big(\frac{\partial^2 \bar{x}}{\partial x^2}\Big) +\partial_{x} \Big( \frac{a^{12}}{a^{22}}\Big)\frac{\partial^2 \bar{x}}{\partial x^2}=\partial_{x}^2 \Big( \frac{a^{12}}{a^{22}}\Big)\frac{\partial \bar{x}}{\partial x},
$$
and vanishes along $y=0$. By the standard method for the first order PDE, we may control its sup norm by the sup norm of the nonhomogeneous term, $\partial_{x}^2 \Big( \frac{a^{12}}{a^{22}}\Big)\frac{\partial \bar{x}}{\partial x}$. Since for all $l \ge 2$ and $\partial^l_x \bar{x} $ vanishes along the boundary $\{  y=0 \}$, again $\varepsilon^2$ will be produced (until it reaches $\varepsilon^{N+2}$)  each time we differentiate $\partial_x \bar{x}$ and $\partial_y \bar{x}$. Since the terms in (\ref{a4.21}) are differentiated for at least $2N+5$ times, we have
\begin{eqnarray*}
&&\varepsilon^{-N-1} \Big( |\partial_{\bar{x}}^{2N+7} \bar{K} |_{L^{\infty}(\mathbb{R}^2_+)} +|\partial_{\bar{x}}^{2N+6} \bar{C} |_{L^{\infty}(\mathbb{R}^2_+)} \\
&& \quad \quad \quad \quad  +|\partial_{\bar{x}}^{2N+6} \bar{D} |_{L^{\infty}(\mathbb{R}^2_+)}+| \bar{y}\partial_{\bar{x}}^{2N+5} \bar{D} |_{L^{\infty}(\mathbb{R}^2_+)} \Big)=O(\varepsilon).
\end{eqnarray*}

With the condition $|w|_{C^{3N+14}} \le \varepsilon^{N+2}$, the part with $l \le 2N+6$ in (\ref{a4.20.5}) becomes
\begin{eqnarray*}
&&\varepsilon^{-N-1} \sum_{|\tau|=0}^{N+3} \sum_{l=3}^{2N+6}\| h^{-\gamma+s-(1+\alpha)(N+3) + |\tau|} \bar{\partial}^{\tau} (\partial_{\bar{x}}^l \bar{K}_{\theta} \eta^{(s-l+2)}_{\theta})\|_{L^2(\partial \bar{\Omega})}\\
&& \quad \quad \le C_s \varepsilon^{-N-1} \sum_{|\tau|=0}^{N+3} \sum_{l=s-2N-4}^{s-1}\|h^{-\gamma+s-(1+\alpha)(N+3)+|\tau|} \bar{\partial}^{\tau} \eta_{\theta}^{(l)}     \|_{L^2(\partial \bar{\Omega})}. \\
\end{eqnarray*}
Then by the (\ref{a4.11}) and mathematical induction, we can sum up the result from $s= 2N+7$ to $s=m$:

\begin{eqnarray*}
&&\sum_{s=2N+7}^{m}\|h^{-\gamma+s} \partial_{\bar{x}} v^{(s)}_{\theta}\| +\| h^{-\gamma+s} \partial_{\bar{y}} v^{(s)}_{\theta} \| +\| h^{-\gamma+s} v^{(s)}_{\theta}\| \\
&&\quad \quad \le \sum_{s=2N+7}^{m}\|h^{-\gamma+s} \partial_{\bar{x}} v^{(s)}_{\theta}\| +\| \sqrt{|b_s|} \partial_{\bar{y}} v^{(s)}_{\theta} \| +\| \sqrt{|b_s|} v^{(s)}_{\theta}\| \\
&& \quad \quad \le C_{m,\gamma}\varepsilon^{-N-1} \Big[ \sum_{|\tau_1+\tau_2|=0}^{m+N+3} \| h^{-\gamma-\alpha m -3(N+3)-5 + |\tau_1+\tau_2|} \Lambda_{\tau_1} \bar{\partial}^{\tau_2} \phi \|_{L^2(\partial \bar{\Omega})}\\
&& \quad \quad \quad +\sum_{|\tau_1+\tau_2|=0}^{m+N+2} \| h^{-\gamma-\alpha m -3(N+3)-5 + |\tau_1+\tau_2|} \Lambda_{\tau_1} \bar{\partial}^{\tau_2} \psi \|_{L^2(\partial \bar{\Omega})} \\
&& \quad \quad \quad +\sum_{|\tau_1+\tau_2|=0}^{m+N+2} \| h^{-\gamma-\alpha m -3(N+3) + |\tau_1+\tau_2|} \Lambda_{\tau_1} \bar{\partial}^{\tau_2} f \| \\
&& \quad \quad \quad + \sum_{l=0}^{2N+6} \sum_{|\tau|=0}^{N+3} \|h^{-\gamma-(1+\alpha)(N+3)+l+2|\tau|} \bar{\partial}^{\tau} \partial_{\bar{x}}^l \eta_{\theta}\|_{L^2(\partial \bar{\Omega}_1)} \Big] \\
&& \quad \quad \quad +\varepsilon \|v_{\theta}\|_{H^{\gamma+1}(\mathbb{R}^2_+)} + C_{m,\gamma}\varepsilon^{-N-1} \Big(1+ \|w \|_{H^{\gamma+2N+13}(\mathbb{R}^2_+)} \Big) \|v_{\theta}\|_{L^{\infty}(\mathbb{R}^2_+)}.
\end{eqnarray*}
After adding the estimate from the case $m \le 2N+6$, it is straightforward to derive an estimate (by the right hand side of our last inequality) on every derivative on $u_{\theta}$ with respect to $\bar{x}$. Solving for the derivatives of $u_{\theta}$ with respect to $\bar{y}$ from our linear equation yields a similar estimate on all the derivatives of $u_{\theta}$. Then we pick $\gamma=m$, change coordinate to $(x,y)$ and choose $\varepsilon$ sufficiently small. Eventually, \lemmaref{l3.1}, \cororef{c3.1} and \lemmaref{l3.3} yield
\begin{eqnarray*}
\|u_{\theta}\|_{H^{(m,m)}(\Omega)}&\le& C_m \varepsilon^{-N-1} \Big( \| \phi\|_{H^{(1+\alpha)m+3N+14}(\partial \Omega)} \\
&&+\| \psi\|_{H^{(1+\alpha)m+3N+13}(\partial \Omega)} + \| f\|_{H^{(1+\alpha)m+3N+10}(\mathbb{R}^2_+)}         \Big) \\
&&+C_m \varepsilon^{-2N-2} \|w\|_{H^{(1+\alpha)m+3N+19}} \Big(\|\phi\|_{H^{2N+13}(\partial \Omega)} \\
&&+\| \psi\|_{H^{2N+12}(\partial \Omega)} + \| f\|_{H^{2N+9}(\mathbb{R}^2_+)}     \Big).
\end{eqnarray*}
\end{proof}

\section{The Nash-Moser iteration with hyperbolic cusps}
In this section, we apply Nash-Moser iteration to solve the Cauchy problem:
\begin{eqnarray*}
\Phi(w)=0 \text{ in } \Omega, \ \ \ \ \ \ w|_{\partial \Omega}=0 \ \ \ \text{ and } \ \partial_{\nu} w|_{\partial \Omega}=\psi.
\end{eqnarray*}
We will assume that the boundary data $\psi$ is in $H_0^{m_0-10}(\partial \Omega_1)$ and that $\| \psi\|_{m_0-10} \le C \varepsilon^{2A}$, where $\| \cdot \|_m := \| \cdot \|_{H^m(\mathbb{R}^2_+)}$ and $A$ is a large parameter which we will determine. Since $\psi$ is the normal derivative of the solution obtained in the bordering elliptic region restricted to the shared boundary, the assumption above is legitimate.

We start with an approximate solution $z_0$ so that $\|\Phi(0)\|_{\rho+1} \le C \varepsilon^{2A}$. Then by the proof of \lemmaref{l4.4}, we may pick $w_0$ such that
$$
w_0|_{\partial \Omega}=0,\ \ \ \ \ \  \partial_{\nu} w_0|_{\partial \Omega}=\psi
$$
and
\begin{equation} \label{a5.1}
\|w_0\|_{H^{(m,m)}(\Omega)} \le \|\psi\|_{(1+\alpha)(\rho+3)+5}\le \| \psi \|_{(m_0-10)} \le C\varepsilon^{2A},
\end{equation}
assuming that $m \le \rho+3$ and
$$
(1+\alpha)(\rho +3)+5 \le m_0-10.
$$
Then by \lemmaref{l3.1}, $w_0$ has an extension to $\mathbb{R}^2_+$, still denoted by $w_0$. With the assumption $m_0-\rho-3 \ge 2N+6$, we compute on $\mathbb{R}^2_+$
\begin{eqnarray} \label{a5.2}
\| \Phi(w_0) \|_{\rho+1} \le C (\varepsilon^{2A}+\|w_0\|_{\rho+3}) \le C \varepsilon^{2A},
\end{eqnarray}
and start the iteration from here. Applying the Taylor expansion theorem to
$$
\mathcal{G}(t) := \Phi(w_0+tu_0)
$$
gives
$$
\Phi(w_0+u_0)=\Phi(w_0) + \mathcal{L}(w_0)u_0 +Q,
$$
where $Q$ is the quadratic error.

Heuristically, because the quadratic error is relatively small, we can make $\Phi(w_0+u_0)$ strictly smaller than $\Phi(w_0)$ by solving $\mathcal{L}(w_0)u_0= -\Phi(w_0)$. We repeat this process and expect $w_0+\sum_{n=0}^{\infty} u_n$ converges to a solution. However, to use the result from the Section \ref{s.8}, we need to modify the standard procedure of Nash-Moser iteration. We define the following:
\begin{eqnarray*}
&&(i)\ \ \ \mu:=\varepsilon^{-\frac{1}{2 \rho}}, \ \ \ \ \mu_n:=\mu^n, \ \ \ \ \text{ and } \ \ S_n :=S_{\mu_n}.\\
&&(ii)\ \ v_n:=S_n w_n, \ \ \ \ \ \text{and } \ \ \theta_n=|\Phi(v_n)|_{C^5}.\\
&&(iii)\ \text{The operator $L_{\theta_n}$ is defined as in (\ref{a4.3}), except being linearized at $v_n$.}\\
&&(iv)\ \ u_n \text{ is the unique solution of }  L_{\theta_n}u_n=f_n  \text{ in } \Omega, \ u_n|_{\partial {\Omega}}=0, \ \partial_{\nu} u|_{\partial {\Omega}}=0,\\
&& \ \ \ \ \ \ \  \text{where $f_n$ is to be stated below.}\\
&&(v)\ \ \ w_{n+1} =w_n + u_n.
\end{eqnarray*}
Then we modify the equation:
\begin{eqnarray} \label{a5.3}
\Phi(w_{n+1})&=&\Phi(w_n)+ \mathcal{L}(w_n) u_n +Q_n(w_n,u_n)\\
&=& \Phi(w_n) +\varepsilon S'_n a^{22}(v_n) L_{\theta_n}(v_n) u_n +e_n, \nonumber
\end{eqnarray}
where $Q_n(w_n,u_n)$ denotes the quadratic error in the Taylor expansion of $\Phi$ at $w_n$, and
\begin{eqnarray*}
e_n &=& (\mathcal{L}(w_n)-\mathcal{L}(v_n))u_n +Q_n(w_n,u_n)\\
&&+ \ \varepsilon (I-S'_n)a^{22}(v_n)L_{\theta_n}(v_n)u_n+\varepsilon \theta_n a^{22}(v_n)\partial^2_{\bar{x}}u_n\\
&& + \ \varepsilon (a^{22}(v_n))^{-1} \Phi(v_n) [ \partial^2_x u_n-\partial_x(loga^{22}(v_n)\sqrt{|g|})\partial_x u_n].
\end{eqnarray*}
We also define
\begin{eqnarray*}
&&(vi)\ \ E_0=0, \ \ \ \ \text{ and } \ \ E_n=\sum_{i=0}^{n-1} e_i.\\ 
&&(vii)\ f_0=-[\varepsilon S'_0a^{22}(v_0)]^{-1}S_0 \Phi(w_0), \ \ \text{ and }\\
&& \quad \quad \ f_n =[\varepsilon S'_n a^{22}(v_n)]^{-1} (S_{n-1} E_{n-1} -S_n E_n +(S_{n-1}-S_n) \Phi(w_0)). \quad \quad
\end{eqnarray*}
Notice that as $u_n$ is in the weighted space over $\Omega$, therefore so is $w_n$. Thus we may extend $w_n$ to $\mathbb{R}^2_+$ so that all the definition above are over $\mathbb{R}^2_+$, except the operator $L_{\theta}$. Then we compute:
\begin{eqnarray*}
&&\text{(i)}\ \ \ \|w_0- v_0\|_m= \|(I-S_0) w_0\|_m \le C \| w_0\|_m \le \varepsilon^{2A}.\\
&&(ii)\ \ \|v_0\|_m = \| S_0 w_0\|_m \\
&&\quad \quad \quad \le \left\{ \begin{array}{ll}
C \|w_0 \|_m \le C \varepsilon^{2A}, & \text{if } m \le \rho+3 ,\\
C \mu_0^{m-\rho-3} \|w_0 \|_{\rho+3}\le C   \varepsilon^{2A},& \text{if } m > \rho+3.
\end{array}\right. \\
&&(iii)\ \|w_0\|_{3N+14} \le \varepsilon^{2A} \text{ and } \|v_0\|_{3N+14} \le \varepsilon^{2A}.\\
&&(iv)\ \ \| \Phi(v_0)\|_m \le C (\varepsilon^{2A} +\|v_0\|_{m+2} ) \le C \varepsilon^{2A}.\\
&&(v)\ \ \ \|f_0\|_m=\|[\varepsilon S'_0a^{22}(v_0)]^{-1}S_0(w_0)\|_m \\
&&\quad \quad \quad \le C \varepsilon^{-1}\Big(\|S'_0 a^{22}(v_0)\|_m \|S_0 \Phi(w_0) \|_2+  \|S'_0 a^{22}(v_0)\|_2 \|S_0 \Phi(w_0) \|_m \Big) \\
&&\quad \quad \quad \le C \varepsilon^{2A-1}.\\
&&(vi)\ \ \|u_0\|_{H^{(m,m)}(\Omega)} \\
&&\quad \quad \quad \le C \varepsilon^{-N-1} \| f_0 \|_{(1+\alpha)+3N+10} + C \varepsilon^{-2N-2}\|v_0\|_{(1+\alpha)m+3N+19} \|f_0\|_{2N+9}\\
&&\quad \quad \quad \le  \varepsilon^{2A-2N-3}.\\
&&(vii)\ \|e_0\|_m \le C \varepsilon^{2A}.
\end{eqnarray*}
We skip the computation of (vii), the proof of (VIII$_n$) below will justify (vii). The seven statements above, together with (\ref{a5.1}) and (\ref{a5.2}), imply the case $n=0$ of the following eight statements, which will prove inductively that the $w_n$ converge to a solution.
\begin{eqnarray*}
&&\ \ \ \ \text{I$_n$ :}\ \|w_n\|_m \le \left\{ \begin{array}{ll}
C_1 \varepsilon^A,& \text{if } (1+\alpha)m+3N+10-\rho  \le -r  ,\\
C_1 \varepsilon^A \mu_n ^{(1+\alpha)m+3N+10-\rho },& \text{if } (1+\alpha)m+3N+10-\rho  \ge r,
\end{array}\right. \\
&&\ \ \quad \quad \quad \quad \quad \quad \quad \quad \quad \quad \quad \quad \quad \quad \quad \quad \quad \quad \quad \quad \quad \quad \quad \quad \quad \text{ for all } m \le \rho+3,\\
&&\ \ \ \text{II$_n$ :}\ \|w_n -v_n \|_m \le C_2 \varepsilon^{A} \mu_n^{m+3N+13+(\alpha-1)\rho}, \quad \ \ \quad \quad \quad \quad\text{ for all } m \le \rho+3,\\
&&\ \ \text{III$_n$ :}\ \|v_n\|_m \le \left\{ \begin{array}{ll}
C_3 \varepsilon^A, & \text{if } (1+\alpha)m+3N+10-\rho  < 0  ,\\
C_3 \varepsilon^A \mu_n ^{m+3N+10+(\alpha-1)\rho },& \text{if } (1+\alpha)m+3N+10-\rho  > 0,
\end{array}\right. \\
&&\ \ \quad \quad \quad \quad \quad \quad \quad \quad \quad \quad \quad \quad \quad \quad \quad \quad \quad \quad \quad \quad \quad \quad \quad \quad \quad \quad \text{ for all } m \in \mathbb{Z}_+,\\
&&\ \ \text{IV$_n$ :}\ \|w_n\|_{3N+14} \le C_1 \varepsilon^{A}, \quad \ \ \ \ \|v_n \|_{3N+14} \le C_3 \varepsilon^{A},\\
&&\ \ \ \text{V$_n$ :}\ \| f_n \|_m \le C_4 \varepsilon^{2A} (1+\mu^{\rho})\mu_{n}^{m-\rho},  \quad \quad \quad \quad \quad \quad \quad \quad \quad \quad \ \text{ for all } m \in \mathbb{Z}_+,\\
&&\ \ \text{VI$_n$ :}\ \|\Phi(w_n)\|_m \le \varepsilon^{A} \mu_n^{m-\rho}, \text{ and } \| \Phi(v_n)\|_m \le C_5 \varepsilon^{A} \mu_n^{m+3N+12+(\alpha-1)\rho },\\
&&\quad \quad \quad \quad \quad \quad \quad \quad \quad \quad \ \ \quad \quad \quad \quad \quad \quad \quad \quad \quad \quad \quad \quad \quad \quad \quad \text{ for all } m \le  \rho+1, \\
&&\ \text{VII$_n$ :}\ \|u_n \|_{H^{(m,m)}(\Omega)} \le \varepsilon^A \mu_n^{(1+\alpha)m+3N+10-\rho}, \quad \quad   \text{ for all } m \le \frac{m_0-3N-21}{1+\alpha},\\
&&\text{VIII$_n$ :}\ \| e_n \|_m \le C_6 \varepsilon^{2A} \mu_n^{m-\rho}, \quad \quad \quad \quad \quad \quad \quad \quad \quad \quad \quad \quad \quad \text{ for all } m \le \rho+1.
\end{eqnarray*}
Above $\rho$ is a big constant integer we will determine.

Now we specify $r$. Notice that
$$
\mathcal{S}=\Big\{ | (1+\alpha)m+3N+10-\rho| \ \Big| \ \ 0 \le m < \frac{m_0-3N-21}{1+\alpha}     \Big\}
$$
is a finite set consisting of strictly positive numbers. Without loss of generality we assume that $\alpha$ is irrational, then $r$ is defined as the minimum of
$$
\mathcal{S} \cup \Big\{\frac{\rho-3N-10}{1+\alpha} -\Big\lfloor \frac{\rho-3N-10}{1+\alpha} \Big\rfloor \Big\} \cup \Big\{\Big\lceil \frac{\rho-3N-10}{1+\alpha} \Big\rceil -\frac{\rho-3N-10}{1+\alpha} \Big\}.
$$
We now assume the statements are valid for all nonnegative integers less than or equal to $n-1$, then prove that they hold for $n$.

I$_n$: Assuming that VII$_{n-1}$, then for all $m \le  \rho+3$:

$$
\|w_n\|_m =\Big\| w_0 + \sum_{i=0}^{n-1} u_i \Big\|_m \le C \varepsilon^{A} \Big(1+\sum_{i=0}^{n-1} \mu_i^{(1+\alpha)m+3N+10-\rho}\Big).
$$
Therefore, if $(1+\alpha)m+3N+10-\rho  \le -r$:
$$
\|w_n\|_m \le 2 \varepsilon^{A}\Big( \sum_{i=0}^{n-1} u_i^{-r} \Big) \le 2 \varepsilon^{A} \sum_{i=0}^{\infty} (\frac{1}{\mu^r})^i \le C_0 \varepsilon^A,
$$
else if $(1+\alpha)m+3N+10-\rho  \ge r$:
\begin{eqnarray*}
\|w_n\|_m &\le& 2 \varepsilon^A \mu_n^{(1+\alpha)m+3N+10-\rho }  \sum_{i=0}^{n-1}\Big( \frac{\mu_i}{\mu_n} \Big)^{(1+\alpha)m+3N+10-\rho }\\
&\le& 2 \varepsilon^A \mu_n^{(1+\alpha)m+3N+10-\rho }  \sum_{i=0}^{n-1}\Big( \frac{1}{\mu^{(1+\alpha)m+3N+10-\rho }} \Big)^{n-i}\\
&\le& 2 \varepsilon^A \mu_n^{(1+\alpha)m+3N+10-\rho }  \sum_{i=1}^{\infty}\Big( \frac{1}{\mu^{r}} \Big)^{i} \le 2C_0 \varepsilon^A \mu_n^{(1+\alpha)m+3N+10-\rho }.
\end{eqnarray*}
Therefore I$_n$ follows by setting $C_1=2C_0$.

II$_n$: Under the condition that $m \le \rho+3$, we deduce from I$_{n-1}$ that
\begin{eqnarray*}
\|w_n-v_n\|_m &=& \|(I-S_n)w_n\|_m \le C \varepsilon^A \mu_n^{m-\rho-3} \|w_n\|_{\rho+3} \\
&\le& C \varepsilon^A \mu_n^{m-\rho} \mu_n^{(1+\alpha)(\rho+3)+3N+10-\rho}\\
&\le& C \varepsilon^A \mu_n^{m+3N+13+(\alpha-1)\rho}.
\end{eqnarray*}
Pick $C_2=C$ above, notice that $C_2$ works for all $n \in \mathbb{N}$.

III$_n$: We further assume $\rho \le m_0 -11$, and if $(1+\alpha)m+3N+10-\rho < 0$, then
$$
m \le \Big\lfloor \frac{\rho-3N-10}{1+\alpha} \Big\rfloor,
$$
because $m$ is an integer. We then compute
$$
\|v_n\|_m=\|S_nw_n\|_m \le C\|w_n\|_{\lfloor \frac{\rho-3N-10}{1+\alpha} \rfloor}.
$$
Because
\begin{eqnarray*}
&&\Big\lfloor \frac{\rho-3N-10}{1+\alpha} \Big\rfloor(1+\alpha)+3N+10-\rho \\
&&=\Big( \Big\lfloor \frac{\rho-3N-10}{1+\alpha} \Big\rfloor+\frac{3N+10-\rho}{1+\alpha}\Big)(1+\alpha)\\
&& \le -r (1+\alpha) \le -r,
\end{eqnarray*}
I$_{n}$ implies $$\|v_n\|_m \le C C_1 \varepsilon^{A} \le C_3 \varepsilon^A. $$
On the other hand, if $(1+\alpha)m+3N+10-\rho > 0$, then again,
$$
m \ge \Big\lceil \frac{\rho-3N-10}{1+\alpha} \Big\rceil,
$$
and we have
\begin{eqnarray*}
\|v_n\|_m &\le& C \mu_n^{  m-\lceil \frac{\rho-3N-10}{1+\alpha} \rceil } \|w_n\|_{\lceil \frac{\rho-3N-10}{1+\alpha} \rceil} \\
&\le& C C_1\varepsilon^{A} \mu_n^{m - \lceil \frac{\rho-3N-10}{1+\alpha} \rceil } \mu_n^{\lceil \frac{\rho-3N-10}{1+\alpha} \rceil (1+\alpha) +3N+10-\rho}\\
&\le& C C_1\varepsilon^A \mu_n^{m+3N+10+\frac{\alpha \rho}{1+\alpha}-\rho + \alpha -\frac{\alpha(3N+10)}{1+\alpha}}\\
&\le& C C_1\varepsilon^A \mu_n^{m+3N+20+(\alpha-1)\rho}\\
&\le& C_3\varepsilon^A \mu_n^{m+3N+20+(\alpha-1)\rho}.
\end{eqnarray*}

IV$_n$: This simply requires $(3N+14)(2+\alpha) < \rho$.

V$_n$: We first estimate with the help of VIII$_{n-1}$:
\begin{eqnarray*}
\|E_n\|_{\rho+1} &\le& \sum_{i=0}^{n-1} \|  e_i \|_{\rho+1}\le C_6 \sum_{i=0}^{n-1} \varepsilon^{2A} \mu_i\\
&\le& C_6 \varepsilon^{2A}\mu_n \sum_{i=1}^{\infty} (\mu^{-i}) \le C_6 \varepsilon^{2A} \mu_n.
\end{eqnarray*}

We then apply Nirenberg inequality to derive
\begin{eqnarray} \label{a5.4}
\|f_n\|_m &\le& C \varepsilon^{-1} ( \|S_{n-1}E_{n-1}-S_nE_n+(S_{n-1}-S_n) \Phi(w_0)\|_m\nonumber \\
&&+ \|S'_n a^{22}(v_n)\|_m \|S_{n-1}E_{n-1}+S_nE_n+(S_{n-1}-S_n)\Phi(w_0)\|_2) \nonumber \\
&\le&C \varepsilon^{-1} ( \|S_{n-1}E_{n-1}-S_nE_n+(S_{n-1}-S_n) \Phi(w_0)\|_m
\nonumber \\
&&+ \mu_n^{m-2}\|a^{22}(v_n)\|_2 \|S_{n-1}E_{n-1}+S_nE_n+(S_{n-1}-S_n)\Phi(w_0)\|_2) \nonumber \\
&\le&C \varepsilon^{-1} ( \|S_{n-1}E_{n-1}-S_nE_n+(S_{n-1}-S_n) \Phi(w_0)\|_m
\nonumber \\
&&+ \mu_n^{m-2} \|S_{n-1}E_{n-1}+S_nE_n+(S_{n-1}-S_n)\Phi(w_0)\|_2).
\end{eqnarray}
Then we prove V$_n$ in three different cases.

\begin{enumerate}
\item $m \ge \rho+1$. Consider
\begin{eqnarray*}
&&\|S_{n-1}E_{n-1}-S_nE_n+(S_{n-1}-S_n) \Phi(w_0)\|_m\\
&&\le \|S_{n-1}E_{n-1}\|_m + \|S_nE_n\|_m + \|S_{n-1} \Phi(w_0)\|_m + \|S_n \Phi(w_0)\|_m\\
&&\le C [\mu_{n-1}^{m- \rho-1}(\|E_{n-1}\|_{\rho+1}+\|\Phi(w_0)\|_{\rho+1})\\
&& \quad \quad \quad \quad +\mu_{n}^{m- \rho-1}(\|E_{n}\|_{\rho+1}+\|\Phi(w_0)\|_{\rho+1}) ] \\
&&\le C \varepsilon^{2A} [ \mu_n^{m-\rho}+\mu_{n-1}^{m-\rho} ]\\
&&\le C \varepsilon^{2A} (1+\mu^{\rho})\mu_n^{m-\rho}.
\end{eqnarray*}
For the other part we compute
\begin{eqnarray} \label{a5.5}
&&\|S_{n-1}E_{n-1}+S_nE_n+(S_{n-1}-S_n)\Phi(w_0)\|_2  \\
&&\le \|S_{n-1}E_{n-1}-E_{n-1}\|_2+\|E_{n-1}-E_n\|_2+\|E_n-S_nE_n\|_2 \nonumber \\
&&\quad \quad +\|S_{n-1}\Phi(w_0)-\Phi(w_0)\|_2+\|\Phi(w_0)-S_n \Phi(w_0)\|_2 \nonumber \\
&&\le C ( \mu_{n-1}^{2-\rho-1} \|E_{n-1}\|_{\rho+1} + \|e_{n-1}\|_2 + \mu_{n}^{2-\rho-1}  \|E_{n}\|_{\rho+1 } \nonumber \\
&&\quad \quad \mu_{n-1}^{2-\rho-1} \|\Phi(w_0)\|_{\rho+1} +  \mu_{n}^{2-\rho-1} \|\Phi(w_0)\|_{\rho+1} ) \nonumber \\
&&\le C \varepsilon^{2A} (1+\mu^{\rho}) \mu_n^{2-\rho}. \nonumber
\end{eqnarray}
Substituting the above into (\ref{a5.4}):
\begin{eqnarray*}
&&\|f_n\|_m \\
&&\quad \le C [ \varepsilon^{2A} (1+\mu^{\rho})\mu_n^{m-\rho}+\varepsilon^A \mu_n^{m-\rho}\varepsilon^{2A} (1+\mu^{\rho}) \mu_n^{2-\rho} ] \\
&&\quad \le C \varepsilon^{2A} (1+\mu^{\rho}) \mu_n^{m-\rho},
\end{eqnarray*}
where the last inequality requires $\rho \ge 2$.
\item $\rho+1 > m \ge 2$. By the computation in (\ref{a5.5}):
$$
\|S_{n-1}E_{n-1}-S_nE_n+(S_{n-1}-S_n) \Phi(w_0)\|_m \le C \varepsilon^{2A} (1+\mu^{\rho})\mu_n^{m-\rho},
$$
the rest is the same as case (1).
\item $m < 2$. As $\|S_n' a^{22}(v_n)\|_{m} \le C$:
$$
\|f_n\|_m \le C \varepsilon^{-1} ( \|S_{n-1}E_{n-1}-S_nE_n+(S_{n-1}-S_n) \Phi(w_0)\|_m).
$$
Similar computation to the previous cases yields the desired estimate.
\end{enumerate}

VI$_n$: From (\ref{a5.3}) and our definition of $f_i$ for $0 \le i \le n-1$, we have
$$
\Phi(w_n) = (I-S_{n-1}) \Phi(w_0) + (I+S_{n-1})E_{n-1}+ e_{n-1}.
$$
We require further that $m \le \rho+1 $, so
\begin{eqnarray*}
&&\| \Phi(w_n)\|_m \\
&&\le C \Big( \mu_{n-1}^{m- \rho-1} \|\Phi(w_0)\|_{\rho+1} +\mu_{n-1}^{m- \rho-1} \|E_{n-1} \|_{\rho+1} + \| e_{n-1} \|_m \Big)\\
&&\le C  \varepsilon^{2A} \mu_{n-1}^{m-\rho} \le C \varepsilon^{2A} (1+\mu^{\rho}) \mu_n^{m-\rho}\\
&&\le \varepsilon^{2A-1} \mu_n^{m-\rho}.
\end{eqnarray*}
Moreover since
\begin{eqnarray*}
\|\Phi(v_n)\|_m &\le& \|\Phi(w_n)\|_m +\|\Phi(w_n)-\Phi(v_n)\|_m\\
&\le&  \varepsilon^{2A-1} \mu_n^{m-\rho}+\varepsilon \|w_n-v_n \|_{m+2}\\
&\le&  \varepsilon^{2A-1} \mu_n^{m-\rho}+C \varepsilon^{A+1} \mu_n^{(1+\alpha)m+3N+12-\rho}\\
&\le& \varepsilon^A (\varepsilon^{A-1}+C\varepsilon) \mu_n^{(1+\alpha)m+3N+12-\rho},
\end{eqnarray*}
VI$_n$ is valid assuming $A \ge 2$.

VII$_n$: We apply \theoremref{t4.5} with $\phi, \psi=0$ to see
\begin{eqnarray*}
\|u_n\|_{H^{(m,m)}(\Omega)} &\le& C [\varepsilon^{-N-1} \|f_n\|_{(1+\alpha)m+3N+10}\\
&&+\quad  \varepsilon^{-2N-2} \|v_n\|_{(1+\alpha)m+3N+19} \|f_n\|_{2N+9}]\\
&\le& C [\varepsilon^{2A-N-1} (1+\mu^{\rho})\mu_n^{(1+\alpha)m+3N+10-\rho}\\
&&+\quad \varepsilon^{3A-2N-2} (1+\mu^{\rho}) \mu_n^{(1+\alpha)m+6N+29+(\alpha-1)\rho} \mu_n^{2N+9-\rho}]\\
&\le&C [\varepsilon^{2A-N-2} \mu_n^{(1+\alpha)m+3N+10-\rho}\\
&&+\quad \varepsilon^{3A-2N-3} \mu_n^{(1+\alpha)m+3N+10-\rho +5N+28 +(\alpha-1)\rho} ]\\
&\le& \varepsilon^{A} \mu_n^{(1+\alpha)m+3N+10-\rho}.
\end{eqnarray*}
The last inequality above requires $A \ge N+3$ and $\rho > \frac{5N+28}{1-\alpha}$.

VIII$_n$: Recall that the error is
\begin{eqnarray*}
e_n &=& Q_n(w_n,u_n)+\varepsilon (I-S'_n)a^{22}(v_n)L_{\theta}(v_n)u_n \\
&&+(\mathcal{L}(w_n)-\mathcal{L}(v_n))u_n+\varepsilon \theta_n a^{22}(v_n)\partial^2_{\bar{x}}u_n\\
&&+\varepsilon (a^{22}(v_n))^{-1} \Phi(v_n) [ \partial^2_x u_n-\partial_x(\text{log}a^{22}(v_n)\sqrt{|g|})\partial_x u_n].
\end{eqnarray*}
We will estimate it term by term.
\begin{enumerate}
\item By Taylor's theorem we have,
$$
Q_n(w_n,v_n)=\int_0^1 (1-t) \frac{\partial^2}{\partial t^2} \Phi(w_n +t u_n) dt.
$$
Then by the Sobolev embedding theorem and the Nirenberg inequality:
\begin{eqnarray*}
\|Q_n(w_n,u_n)\|_m &\le& \int_0^1 C (\|\nabla^2 \Phi (w_n+tu_n) \|_2)\|u_n\|_4\|u_n\|_{m+2}\\
&& \quad + \|\nabla^2 \Phi (w_n+tu_n) \|_{m+2} \|u_n\|_4.
\end{eqnarray*}
Then by (iii) of \lemmaref{l3.4}:
\begin{eqnarray*}
\|Q_n(w_n,u_n)\|_m &\le& C [(1+\|w_n\|_4+\|u_n\|_4)\|u_n\|_4\|u_n\|_{m+2}\\
&&\quad + (1+\|w_n\|_{m+2}+\|u_n\|_{m+2})\|u_n\|_4^2]\\
&\le& C \varepsilon^{2A} (\mu_n^{m-\rho+6N+33+(\alpha-1)\rho}+\mu_n^{m-\rho+(\alpha-2)\rho +9N+53})\\
&\le& C \varepsilon^{2A} \mu_n^{m-\rho},
\end{eqnarray*}
where we have used $m\le \rho+1$, $\alpha <1$, $\frac{9N+53}{2-\alpha} \le \rho$ and $\frac{6N+33}{1-\alpha} \le \rho$.

\item Recall that $\mu^\rho=\varepsilon^{-\frac{1}{2}}$
\begin{eqnarray*}
&&\|\varepsilon (I-S'_n)a^{22}(v_n)L_{\theta}(v_n)u_n\|_m\\
&\le& C \varepsilon [ \|a^{22}(v_n) \|_m \|f_n\|_2 + \|a^{22}(v_n)\|_2 \|f_n\|_m  ]\\
&\le& C \varepsilon [\|v_n\|_{m+2} \|f_n\|_2 + \|f_n\|_m          ]\\
&\le& C \varepsilon [\varepsilon^A \mu_n^{m+3N+10+(\alpha-1)\rho} \varepsilon^{2A} (1+\mu^\rho) \mu_n^{2-\rho}  + \varepsilon^{2A}(1+\mu^{\rho}) \mu_n^{m-\rho}                                          ]\\
&\le& C \varepsilon^{2A} \mu_n^{m-\rho}.
\end{eqnarray*}
For the last inequality we need $\frac{3N+12}{1-\alpha} \le \rho$.
\item We first recall that $\mathcal{L}(w)u=\varepsilon a^{ij}u_{;ij}+2 \varepsilon^5 K |g| \langle \nabla_g z, \nabla_g u \rangle.$ So a direct computation together with the Nirenberg inequality yeilds,
\begin{eqnarray*}
&&\|(\mathcal{L}(w_n)-\mathcal{L}(v_n))u_n\|_m \\
&\le& C \varepsilon (\|w_n-v_n\|_m \|u_n\|_4+\|w_n-v_n\|_4\|u_n\|_{m+2})\\
&\le& C \varepsilon ( \varepsilon^A \mu_n^{m+3N+13+(\alpha-1)\rho}
 \varepsilon^A \mu_n^{4(1+\alpha)+3N+10-\rho}\\
&&\quad + \varepsilon^A \mu_n^{4+3N+13+(\alpha-1)\rho} \varepsilon^A \mu_n^{(1+\alpha)m+3N+10-\rho} \\
&\le& C \varepsilon^{2A+1}(\mu_n^{m-\rho+6N+31+(\alpha-1)\rho}+\mu_n^{m+\alpha m +(\alpha-1)\rho+6N+23-\rho})\\
&\le& C \varepsilon^{2A+1}(\mu_n^{m-\rho+6N+31+(\alpha-1)\rho}+\mu_n^{m-\rho+ 6N+24-(1-2\alpha)\rho})\\
&\le& C \varepsilon^{2A+1} \mu_n^{m-\rho},
\end{eqnarray*}
where we have used $m \le \rho+1$, $\frac{6M+31}{1-\alpha} \le \rho$, $\alpha < \frac{1}{2}$ and $\frac{6N+24}{1-2 \alpha} \le \rho$.
\item With the Sobolev embedding theorem and the Nirenberg inequality:
\begin{eqnarray*}
&&| \varepsilon \theta_n a^{22}(v_n) \partial_{\bar{x}}^2 u_n \|_m \\
&\le& \varepsilon \theta_n (\|u_n\|_2 + \|v_n\|_{m+2}\|u_n\|_2 + \|u_n\|_{m+2}+\|v_n\|_4 \|u_n\|_{m+2})\\
&\le& \varepsilon^{A+1} \mu_n^{4+3N+13+(\alpha-1)\rho}  \varepsilon^A \mu_n^{2(1+\alpha)+3N+10-\rho} \\
&&+ \varepsilon^{A+1} \mu_n^{4+3N+13+(\alpha-1)\rho} C \varepsilon^A \mu_n^{m+3N+12+(\alpha-1)\rho} \varepsilon^A \mu_n^{2(1+\alpha)+3N+10-\rho} \\
&&+ \varepsilon^{A+1} \mu_n^{4+3N+13+(\alpha-1)\rho} \varepsilon^A \mu_n^{(1+\alpha)(m+2)+3N+10-\rho}  \\
&&+ \varepsilon^{A+1} \mu_n^{4+3N+13+(\alpha-1)\rho} C \varepsilon^A \mu_n^{4+3N+10+(\alpha-1)\rho} \varepsilon^A \mu_n^{(1+\alpha)(m+2)+3N+10-\rho} \\
&\le& C \varepsilon^{2A} \mu_n^{m-\rho}.
\end{eqnarray*}
The last inequality requires $m \le \rho+1$, $\alpha < \frac{1}{2}$, $\frac{6N+31}{1-\alpha} \le \rho$, $\frac{6N+32}{1-2\alpha} \le \rho$ and $\frac{9N+46}{2-3\alpha} \le \rho$.
\item Again with the Sobolev embedding theorem and the Nirenberg inequality:
\begin{eqnarray*}
&&\| \varepsilon (a^{22}(v_n))^{-1} \Phi(v_n) \partial_x^2 u_n \|_m \\
&\le& \varepsilon ( \|a^{22}(v_n)\|_m \|\Phi(v_n)\|_2 \|u_n\|_4 \\
&&\quad +\|a^{22}(v_n)\|_2 \| \Phi(v_n) \|_m \|u_n\|_4 + \|a^{22}(v_n)\|_2 \|\Phi(v_n)\|_2 \|u_n\|_{m+2}\\
&\le& C \varepsilon^{2A} \mu_n^{m-\rho},
\end{eqnarray*}
where we need $m \le \rho+1$, $\alpha < \frac{1}{2}$, $\frac{9N+44}{2-3\alpha}\le \rho$ and $\frac{6N+33}{1-2\alpha}$.

\item Under the same conditions, using the computation as in case (5) implies
$$
\|\varepsilon (a^{22})^{-1} \Phi(v_n) \partial_x ( \text{log}a^{22}(v_n) \sqrt{|g|} ) \partial_x u_n   \|_m\le C \varepsilon^A \mu_n^{m-\rho}.
$$
\end{enumerate}
Therefore, the discussion above concludes the main result in this section,
\begin{coro} \label{c5.1}
Suppose $m_0 \ge \frac{1+ \alpha}{1- 2\alpha}(6N+33)+4N+33$ and $A \ge N+3 $, and write $\chi := \lfloor \frac{m_0-13}{(1+\alpha)^2}-\frac{5N+21}{1+\alpha} \rfloor$ by, then $w_n \rightarrow w \in H^{\chi}$ and $\Phi(w_n) \rightarrow 0$ in $C^0(\Omega)$.
\end{coro}
\begin{proof}
Pick $\rho=\lfloor \frac{m_0-13}{1+\alpha} \rfloor -2N-9$, so that the conditions we need on $\rho$ in the iteration procedure are valid, including $m_0-\rho-3 \ge 2N+6$, $(1+\alpha)(\rho+3)+5 \le m_0-10$ and $\rho \ge \frac{6N+
33}{1-2\alpha}$. Moreover, it is easy to check $(1+\alpha)\chi +3N +10 -\rho \le -1$. Thus
$$
\|w_i-w_j\|_{\chi} \le \sum_{n=j}^{i-1}\|u_n\|_{\chi} \le \varepsilon^{N+3}\sum_{n=j}^{i-1} \mu_n^{(1+\alpha)\chi+3N+10-\rho} \le \varepsilon^{N+3} \sum_{n=j}^{i-1} \mu^{-n}.
$$
Hence the $w_n$ form a Cauchy sequence in $H^{\chi}(\Omega)$. Also,
$$
\|\Phi(w_n)\|_{C^0(\Omega)} \le C \| \Phi(w_n) \|
_2 \le \varepsilon^{N+3} \mu_n^{2- \rho},
$$
whose limit is  $0$ as $n \rightarrow \infty$ because $\rho \ge 2$.
\end{proof}

For a solution of (\ref{a1}) in a neighborhood of the origin, we need to obtain a solution in the complement of $\Omega_1^- \cup \Omega_2^-$. The complement consists of domains whose boundary are at least Lipschitz. However, we cannot apply the result in \cite{MR2765727} directly, since they dealt with sectors, whose angle can be made small by a linear transformation. Instead, we first apply a coordinate change,
$$
\tilde{x}=x^{\frac{1}{3}},\ \ \ \ \ \  \tilde{y}=y.
$$
The assumption $\bar{\alpha} < 1$ implies this coordinate change map $\Omega_{\kappa}^+$ to a cusp domain. Then we may use the method in the previous sections to derive the same estimate. Therefore as in \cite{MR2765727},
\begin{pro} \label{p5.2}
If $m_0 \ge 2N+9$, then we have a sequence,
$$
\{ \ w_n \ |\ \ w_n \in \overline{H}^{m_0-8}(\Omega) \ \ \ \text{and} \ \ \  w_n|_{\partial \Omega}=0 \ \},
$$
such that $w_n \rightarrow w$ in $\overline{H}^{m_0-8}(\Omega)$, with $$\|w\|_{m_0-8} \le C \varepsilon^{2N+6}.$$  Furthermore, $\Phi(w_n) \rightarrow 0$ in $C^0(\Omega)$.
\end{pro}

To prove the main theorem, on each elliptic region $\Omega_{\kappa}^+$ we construct $w_{\kappa}^+ \in \overline{H}^{m_0-8}(\Omega_{\kappa}^+)$ which satisfies,
\begin{eqnarray*}
&&\Phi(w_{\kappa}^+)=0 \ \text{ in } \ \Omega_{\kappa} \text{ , } \ \ \ \ \ \ w_{\kappa}^+=0 \ \text{ on } \ \partial \Omega_{\kappa}^+.
\end{eqnarray*}
Then by the  trace theorem
$$
\partial_{\nu} w_{\kappa}^+ \in H^{m_0-10}(\partial \Omega_{\kappa}^+),
$$
which satisfies the conditions of \cororef{c5.1}. Therefore we may use \cororef{c5.1} to obtain $w_{\varrho}^-\in H^{\chi}(\Omega_{\varrho}^-)$ (with $\chi$ defined in \cororef{c5.1}) satisfying
\begin{eqnarray*}
&&\Phi(w_{\varrho}^-)=0 \text{ in } \Omega_{\varrho}^-, \ \ \ \ w_{\varrho}^-|_{\partial \Omega_{\varrho}^-} =0 \ \ \text{ and } \ \partial_{\nu}w_{\varrho}^-|_{\partial \Omega_{\varrho}^-}=\partial_{\nu}w_{\kappa}^+|_{\partial \Omega_{\varrho}^-}.
\end{eqnarray*}
Lastly, we want to patch individual solutions to a complete one in a neighborhood of the origin. Consider any pair of $w_{\kappa}^+$ and $w_{\varrho}^-$ sitting on $\Omega_{\kappa}^+$ and $\Omega_{\varrho}^-$ respectively. By the choice of the approximate solution, the fact that $w_{\kappa}^+$ and $w_{\varrho}^-$ are in weighted spaces and $|h|_{C^1(|y|\le \sigma)}$ decreases to zero as the small positive parameter $\sigma \rightarrow 0$, the common boundary of any pair of $\Omega_{\kappa}^+$ and $\Omega_{\varrho}^-$ is noncharacteristic. Together with the trace theorem and the Sobolev embedding theorem we deduce that $w_{\kappa}^+$ and $w_{\varrho}^-$ agree along the boundary up to order $\chi-2$. This assures the combined solution is $C^{\chi-2}$ smooth across the boundary. Therefore we have constructed a solution in a neighborhood of the origin.


\bibliographystyle{abbrvurl}
\bibliography{bib}
\end{document}